\documentclass[12pt,reqno,a4paper]{amsart}
\usepackage[normalem]{ulem}
\usepackage{yfonts, amsfonts}
\usepackage[centertags]{amsmath}
\usepackage{amsmath, amsopn, amsthm, amssymb}
\usepackage{a4wide}
\usepackage{graphicx, subfigure}
\usepackage{times}
\usepackage{color}
\usepackage[mathcal]{euscript}
\usepackage{mathtools}
\usepackage{bbold}
\usepackage{hyperref}
\usepackage{mathrsfs}
\newtheorem{proposition}{Proposition}[section]

\theoremstyle{definition}
\newtheorem{theorem}{Theorem}[section]%

\newtheorem{corollary}[proposition]{Corollary}
\newtheorem{lemma}[proposition]{Lemma}

\theoremstyle{remark}
\newtheorem{remark}{Remark}
\newcommand{\mL}{\mathcal L}

\newcommand{\mF}{\mathcal {F}}
\newcommand{\mP}{\mathcal {P}}
\newcommand{\C}{\mathcal {C}}
\newcommand{\D}{\mathcal {D}}
\newcommand{\G}{\mathcal {G}}
\newcommand{\vp}{\varphi}

\newcommand{\PP}{\mathbb {P}}
\newcommand{\NN}{\mathbb {N}}
\newcommand{\eps}{\varepsilon}

\DeclareMathOperator{\esssup}{ess sup}

\DeclareMathOperator{\leb}{\lambda}

\DeclareMathOperator{\diff}{Diff}

\begin{document}
\title{Quenched decay of correlations for nonuniformly hyperbolic random maps with an ergodic driving system}
\author{Jos\'e F. Alves}
\address{Centro de Matem\'atica da Universidade do Porto, Rua do Campo Alegre 687, 4169-007 Porto, Portugal}
\email{jfalves@fc.up.pt}

\author{Wael Bahsoun}

\address{Department of Mathematical Sciences, Loughborough University,
Loughborough, Leicestershire, LE11 3TU, UK}
\email{W.Bahsoun@lboro.ac.uk}

\author{Marks Ruziboev}
\address{Faculty of Mathematics, University of Vienna, Oskar Morgensternplatz 1, 1090 Vienna, Austria}
\email{marks.ruziboev@univie.ac.at}
\author{Paulo Varandas}
\address{Centro de Matem\'atica da Universidade do Porto, Rua do Campo Alegre 687, 4169-007 Porto, Portugal}
\email{pcvarand@gmail.com}

\date{\today}
\thanks{The research of W. Bahsoun (WB) is supported by EPSRC grant EP/V053493/1. WB would like to thank the hospitality of the University of Vienna where part of this work was carried. M. Ruziboev (MR) research is supported by the Austrian Science Fund (FWF): M2816 Meitner Grant. MR would like to thank the hospitality of Loughborough University where this project was initiated. Both WB and MR would like to thank Henk Bruin for useful discussions and would like to thank the University of Porto for its hospitality during their visit to J.F. Alves and P. Varandas. J.F. Alves was partially supported by CMUP (UID/MAT/00144/2019),
PTDC/MAT-PUR/4048/2021 and PTDC/MAT-PUR/28177/2017, which are funded
by FCT (Portugal) with national (MEC) and European structural funds
through the program  FEDER, under the partnership agreement PT2020. The authors are grateful for anonymous reviewer for careful reading of the paper and for useful comments and suggestions.}
\begin{abstract} 
In this article we study random tower maps driven by an \emph{ergodic} automorphism. We prove quenched exponential correlations decay for tower maps admitting exponential tails. Our technique is based on constructing suitable cones of functions, defined on 
the random towers, which contract with respect to the Hilbert metric under the action of appropriate transfer operators. We apply our results to obtain quenched exponential correlations decay for several \emph{non-iid} random dynamical systems including small random perturbations of Lorenz maps and Axiom A attractors.   
\end{abstract}
\maketitle
\markboth{J. F. Alves, W. Bahsoun, M. Ruziboev, P. Varandas}{Quenched decay of correlations for non-uniformly hyperbolic maps}


\section{Introduction}
Tower extensions are flexible tools that are used to study important statistical  properties, such as the rate of the correlation decay, for non-uniformly hyperbolic systems. For deterministic dynamical systems\footnote{For general references on the basics of random dynamical systems and their ergodic theory we refer the reader to \cite{A98,K,LQ95}.} such tools are known as Young towers after the pioneering work of L-S Young \cite{Y98,Y99}. For random dynamical systems they are called random Young towers and they were first studied in \cite{BBMD} for independent identically distributed (iid) non uniformly expanding random systems that mix exponentially. Recently in \cite{BBR19} such random towers were developed for non-uniformly expanding iid random systems that mix polynomially. More recently, random towers for small random perturbations of iid partially hyperbolic attractors were constructed in \cite{ABR22}. In fact, the iid requirement in \cite{ABR22} is only needed to exploit the results of \cite{BBMD, BBR19} when quotienting the random hyperbolic tower to an expanding one. Research on statistical properties for random dynamical systems has been one of the most active directions in ergodic theory, see for instance \cite{ABR22,AFGV,Bu,DFGV19,DH20, HKRVZ22, NPT21, SSV21, SVZ}; in particular, on probabilistic limit theorems \cite{H22,Su22} for systems admitting the random tower extensions studied in \cite{ABR22, BBR19, BBMD}. 

In this article we drop the iid assumption in \cite{ABR22, BBMD} and study, for the first time, random tower maps driven by an \emph{ergodic} automorphism. Under the assumption that the tower maps have exponential tails, we show that the associated random dynamical systems admit quenched exponential decay of correlations. Unlike \cite{BBMD, BBR19} where coupling techniques that exploited the iid property in a non-trivial way were employed\footnote{See for instance Section 7 in \cite{BBR19} where iid is explicitly used.}, in this paper our technique is based on constructing suitable cones of functions defined on random towers. We show that such cones contract, with respect to the Hilbert metric, under the action of appropriate random transfer operators associated with the random tower maps. In contrast with the deterministic setting\footnote{For more information on this technique in the determinstic setting, we refer the reader to the work of Liverani \cite{Liv1, Liv2} where it was applied to Anosov systems, and to piecewise uniformly expanding maps of the interval. While the reference \cite{MD01} includes the first application of this technique to deterministic Young towers.} \cite{MD01} where the exponential decay of correlation becomes a consequence, in the random setting proving that iterates of the transfer operators, acting on observables in the cones, converge exponentially fast to the equivariant family of densities requires some effort\footnote{This is not an unexpected difficulty. See the pioneering work of \cite{Bu} where the Birkhoff cone technique was applied for expanding on average non-iid random Lasota-Yorke maps. In particular, see Lemma 3.4 of \cite{Bu}.}. This is due to the fact that in the random setting the sequence of instants at which one observes a contraction in the Hilbert metric is determined by a random variable (see Lemma \ref{lem:mix} and Proposition~\ref{prop:key}). To overcome this hurdle, we prove that such `good instances' of contraction are not sparse along the orbits of the base map, thanks to its ergodicity (see Lemmas \ref{lem:erg} and \ref{lem:erg2}). This enables us to get quenched exponential decay rates for observables in the cone, and then pass such rates H\"older observables on the tower. We apply our results and prove quenched exponential correlations decay for several \emph{non-iid} random dynamical systems including small random perturbations of Lorenz maps and Axiom A attractors. 

The paper is organised as follows. Section \ref{sec:set} includes the setup for a random tower with ergodic driving. Moreover, it includes the statements of the main results of the paper, Theorem \ref{thm:main0} and Theorem \ref{thm:main1}. Section \ref{sec:app} is devoted for applications of Theorem \ref{thm:main0} and Theorem \ref{thm:main1} to small random perturbations of Lorenz maps and Axiom A attractors. Finally, in Section \ref{sec:proof1} we provide a proof for Theorem \ref{thm:main0} and in section \ref{sec:proof2} we provide a proof for Theorem \ref{thm:main1}.
\section{The set up}\label{sec:set}
A random dynamical system on a measure space $(X, \leb)$  over a measure preserving system $(\Omega, \PP, \sigma)$ is described by a skew product
\begin{equation}\label{eq:skew1}
S:\Omega\times X\to \Omega\times X, \quad (\omega, x)\mapsto (\sigma\omega, f_\omega(x)). 
\end{equation}
Here we aim to study dynamics of $f_\omega:X\to X$, for $\PP$-almost every $\omega\in\Omega$. We assume that $\sigma: \Omega \to \Omega$ is an \emph{ergodic} automorphism of a probability space $(\Omega, \PP)$ and $(\omega, x)\mapsto f_{\omega}(x)$ is a measurable function on $\Omega\times X$. The random orbits are obtained by iteration of the maps
\begin{equation}\label{eq:randomor}
f^{n}_\omega=f_{\sigma^{n-1}\omega}\circ \cdots\circ f_{\sigma\omega}\circ f_\omega, \qquad n\geq 1.
\end{equation}
The abstract framework in this work is called a random tower map. It is defined as follows: let  $\{\Lambda_\omega\}_{\omega\in\Omega}$ be a measurable family, in the sense of \cite[Section~3]{CF94}, of measurable subsets of $X$ and assume $\leb_\omega(\Lambda_\omega)=1$ for almost every $\omega\in\Omega$. Here $\leb_\omega$ denotes the restriction of the reference measure on $X$ to the subset $\Lambda_\omega$. 

We say that 
the random dynamical system admits a random tower   
if for $\PP$-almost every $\omega\in\Omega$  there exists a countable partition $\{\Lambda^j(\omega)\}_j$ of $\Lambda_\omega$ and a measurable return time function $R_\omega:\Lambda_\omega\to \mathbb N$ that is constant on each $\Lambda^j(\omega)$ and if $R_\omega|\Lambda^j(\omega)=n$ then   $f^{n}_\omega(x)\in\Lambda_{\sigma^n\omega}$ for $\PP$-almost every $\omega\in\Omega$ and $\leb_\omega$ -almost every $x\in\Lambda_\omega.$ 
The tower 
$$
\Delta_\omega = \Big\{ (x,\ell) \in X \times \mathbb Z_+ \colon x\in \bigcup_{j} \Lambda^j(\sigma^{-\ell}\omega)\; 
\text{and}\; 0\le \ell \le R_{\sigma^{-\ell} \omega}(x)-1 \Big\}
$$ 
at $\omega\in\Omega$ is formed by a collection of floors $\Delta_{\omega, \ell}$, $\ell\ge 0$, where $\Delta_{\omega, 0}=\Lambda_\omega\times \{0\}$ and 
\begin{equation*}
\Delta_{\omega, \ell}=\left\{(x, \ell)\mid x \in \Lambda_{\sigma^{-\ell}\omega},   R_{\sigma^{-\ell}\omega}(x)> \ell\right \} \text{ for } \ell\ge 1.
\end{equation*} 
The random tower map $F_\omega: \Delta_\omega\to \Delta_{\sigma\omega}$ is given by
\begin{equation}
F_\omega(x, \ell)=\begin{cases} 
(x, \ell+1), \quad\text{if}\quad \ell+1<R_{\sigma^{-\ell}\omega}(x) \\
(f^{\ell+1}_{\sigma^{-\ell}\omega}x, 0), \quad\text{if}\quad \ell+1=R_{\sigma^{-\ell}\omega}(x).
\end{cases}
\end{equation}
We assume that $F_\omega$ is non-singular with respect to $\lambda_\omega$. The collection $\Delta=\{\Delta_{\omega}\}_{\omega\in\Omega}$ is called a random tower.  The  fibered system $\{F_\omega\}_{\omega\in\Omega}$ on $\Delta$ is called the \emph{random tower map}.

Notice that the partition $\{\Lambda^j(\sigma^{-k}\omega)\}$, $k\ge 1$, induces a partition on the $k$th level of the tower $\Delta_\omega$. Thus, we have a partition $\mathcal P_\omega=\{\Delta_{\omega, \ell}^j\}$  on every  $\Delta_\omega$ induced by $\{\Lambda^j(\sigma^{-k}\omega)\}_j$, $k\ge 1$.   
For $(x,\ell) \in \Delta_\omega$, let $\hat R_\omega$ be an extension of $R_\omega$ to  the tower $\Delta_{\omega}$ defined as 
\begin{equation}\label{eq:hatR}
\hat R_\omega(x, \ell)=R_{\sigma^{-\ell}\omega}(x)\quad\text{and}\quad \hat R_{\omega, \ell}^j=\hat R_\omega|\Delta_{\omega, \ell}^j.
\end{equation}
The reference measure $\leb_\omega$ and the $\sigma$-algebra on $\Lambda_\omega$ naturally lifts to $\Delta_\omega$ and by abuse of notation we denote it by $\leb_\omega$. The lifted $\sigma$-algebra will be denoted by $\mathcal B_\omega$. 
For $z\in\Delta_\omega$, let 
\begin{equation}
    R_1(\omega, z)=\hat R_\omega(z)\text{ and }R_n(\omega, z)=  R_1 (\sigma^{R_{n-1}}\omega, F^{R_{n-1}}_\omega(z))+ R_{n-1}(\omega, z) \text{ for }n\ge2.
\end{equation} 

\medskip
Next we define the separation time $s:\Delta\times\Delta\to \mathbb Z_+\cup\{\infty\}$ for $\PP$-almost every $\omega$ by setting $s((\omega_1, z_1), (\omega_2, z_2)) =0$ {if} $\omega_1\neq\omega_2$ (i.e. $z_1$ and $z_2$  are in different fibers) and, if $\omega_1=\omega_2=\omega$ and $z_1, z_2\in\Delta_{\omega, 0}$ then we set $s((\omega, z_1), (\omega, z_2)) =s_\omega(z_1, z_2)$ to be the largest $n\ge 0$ such that $F_\omega^n(z_1)$ and $F_\omega^n(z_2)$  belong to the same element of $\mathcal P_{\sigma^n\omega}$. One can check that $d((\omega_1, z_1), (\omega_2, z_2))=\gamma^{s((\omega_1, z_1), (\omega_2, z_2))}$ is a distance. When $\omega_1=\omega_2$ we also use the notation $d_\omega(z_1, z_2)=\gamma^{s_\omega(z_1, z_2)}$ and $F_\omega^{-1}=(F_\omega)^{-1}$.

We assume that the random tower satisfies the following properties. 
\begin{itemize}
\item[(P1)] \textbf{Markov}:  $F_\omega|\Delta_{\omega, \ell}^j:\Delta_{\omega, \ell}^j\to \Delta_{\sigma\omega, 0}$ is a measurable bijection whenever $\Delta_{\omega, \ell}^j=\Delta_{\omega, \ell}\cap F^{-1}_\omega(\Delta_{\sigma\omega, 0})$ is non-empty;
\item[(P2)] \textbf{Bounded distortion}: There are constants $D>0$ and $0<\gamma<1$ such that for $\PP$-almost every $\omega$ and each $\Delta_{\omega, 0}^j$  the corresponding Jacobian $JF_\omega$ is positive and   satisfies 
\begin{align}
\label{bddd}
&\left|\frac{JF_\omega(x)}{JF_\omega(y)}-1\right|\le D\gamma^{s_\omega(F_\omega(x), F_\omega(y))}, \text{ for all  } x,y\in \Delta_{\omega, 0}^j;
\end{align}
\item[(P3)] \textbf{Weak expansion}: $\mathcal P_\omega$ is a generating partition for $F_\omega$ i.e. diameters of the partitions $\vee_{j=0}^n F^{-j}_{\omega}\mathcal P_{\sigma^{j}\omega}$ converge to zero as $n$ tends to infinity;
\item[(P4)] \textbf{Aperiodicity}: There are $N\in\mathbb N$ and 
$\{t_i\in\mathbb Z_+\mid i=1, 2, ..., N\}$   such that g.c.d.$\{t_i\}=1$ and $\epsilon_i >0$ so that
for $\PP$-almost every $\omega \in \Omega$ and $i = 1, 2, \dots, N$ we have
$\leb_\omega\{x\in\Lambda_\omega\mid R_\omega(x)=t_i\}>\epsilon_i$.
\item[(P5)] \textbf{Return time asymptotics}: There are constants $C, \theta>0$ such that 
 $$\leb_\omega\{x\in\Delta_\omega\mid R_\omega(x)>n\}\le C e^{-\theta n}.$$
\end{itemize}

\subsection{Statement of the main results}
This subsection is devoted to the statement of the main results. In order to do so, we need to recall some notations and concepts.
Recall that  $\mathcal B_\omega$ is the  Borel $\sigma$-algebra on $\Delta_\omega$. We consider also the space of Lipschitz continuous observables, defined as follows.
Let $0<\gamma< 1$ be as in (P2) above. 
For each $\vp:\Delta\to \mathbb R$  define 
\[
 |\vp_\omega|_\infty =\sup_{x\in\Delta_\omega}|\vp_\omega(x)|, \quad
 |\vp_\omega|_h=\sup_{\ell, j}\sup_{x,y\in\Delta_{\omega, \ell}^j}\frac{|\vp_\omega(x)-\vp_\omega(y)|}{d_\omega(x, y)},  
\]
where $\phi_\omega$ is the restriction of $\phi$ to $\Delta_\omega$.
Let 
$$\|\vp_\omega\|=|\vp_\omega|_h+|\vp_\omega|_\infty.$$ Consider the spaces 
$$L^\infty=\{\vp:\Delta\to \mathbb R\mid\,\, \esssup_{\omega\in\Omega}|\vp_\omega|<+\infty\}$$
and
$$\mF=\{\vp:\Delta\to \mathbb R\mid\,\, \esssup_{\omega\in\Omega}\|\vp_\omega\|<+\infty\}.$$

The dynamics of a random tower map $\{F_\omega\}_\omega$ over a measure preserving system $(\Omega, \PP, \sigma)$ can be modeled by the skew product
\begin{equation}\label{eq:skew}
F: \mathbb \Delta\to \mathbb \Delta, \quad\text{given by } F(\omega, x) = (\sigma\omega, F_\omega(x)), 
\end{equation}
where 
$$\mathbb \Delta=\{(\omega, x)\mid \omega\in\Omega, x\in\Delta_\omega\}.$$

Let ${\mathcal M}_{\hat S}(\Omega\times X)$ denote the space of \emph{random probability measures} $\mu$ on 
$\Omega\times X$ whose marginal on $\Omega$ coincides with $\PP$ (this set is always non-empty, cf. Corollary 6.13 in \cite{Cra02}). In particular, if $\mu \in {\mathcal M}_{\hat S}(\Omega\times X)$ then
\begin{enumerate}
\item[(a)] $\mu$ admits a disintegration
$
d\mu(\omega,x)=d\mu_\omega(x)d\PP(\omega),
$
where $\mu_\omega$ are called the \emph{sample measures} of $\mu$;
\item[(b)] $\omega \mapsto \mu_\omega$ is measurable;
\item[(c)]  the family $\{\mu_\omega\}_\omega$ is \emph{equivariant}: $(F_\omega)_*\mu_\omega=\mu_{\sigma\omega}$
for $\PP$-almost every $\omega\in \Omega$.
\end{enumerate}
We abuse notation and will write  $\mu=\{\mu_\omega\}_\omega$ to denote a random probability measure.
Our first result is that random tower maps admit a unique fibrewise mixing absolutely continuous random probability measures.

\begin{theorem}\label{thm:main0} 
Assume that $\{F_\omega\}_{\omega\in\Omega}$ is a random tower map on $\Delta$ satisfying (P1)-(P5).
There exists a unique random probability measure  $\mu=\{\mu_\omega\}_\omega$, $\omega\to\mu_\omega$ measurable, 
such that $\mu_\omega\ll \leb_\omega$. Moreover, there exists $C>0$ (independent of $\omega$) such that $1/C<\frac{d\mu_\omega}{d\leb_\omega}<C$
for $\PP$-almost every $\omega\in \Omega$. Furthermore, this family is fibrewise mixing:
\begin{equation}\label{eq:fibmix}
\lim_{k\to\infty}\left|\mu_\omega(A_\omega\cap F^{-k}_\omega(A'_{\sigma^k\omega})) -\mu_\omega(A_\omega)\mu_{\sigma^k\omega}(A'_{\sigma^k\omega})\right|=0
\end{equation}
for any $A_\omega\in\mathcal B_\omega$ and  $A_{\sigma^k\omega}'\in\mathcal B_{\sigma^k\omega}$.
\end{theorem}

\begin{remark} It is essential for our decay of correlation technique to obtain measures with, $1/C<\frac{d\mu_\omega}{d\leb_\omega}<C$ for $C>0$ independent of $\omega$. Although there are many results in the literature showing that such a family exists and it satisfies the above boundedness property, the measurability $\omega\to\mu_\omega$ is missing from the literature. Such a property is required to consider integrals of the form $\int_\Omega\mu_\omega(A) dP(\omega)$, which are in turn needed to obtain the fibrewise mixing property. Below in Section \ref{sec:proof1} we provide a proof showing that $\omega\to\mu_\omega$ is indeed measurable. 
\end{remark}
The main result of the paper is the following statement concerning the `operational' quenched exponential decay of correlations:

\begin{theorem}\label{thm:main1} 
Assume that $\{F_\omega\}_{\omega\in\Omega}$ is a random tower map on $\Delta$ satisfying (P1)-(P5), and
let $\mu=\{\mu_\omega\}$ be the unique random probability measure whose sample measures are absolutely continuous with respect to $\lambda_\omega$. There are $\beta\in(0,1)$ and $C_\omega>0$, finite for $\PP$-almost every $\omega$, such that for every $\psi=(\psi_\omega)_\omega \in L^{\infty}$, 
$\vp_\omega \in \mF$ and $n\ge 1$,
$$\left| \int_{\Delta_{\omega}} (\psi_{\sigma^n\omega}\circ F^{n}_\omega)\vp_\omega d\leb_{\omega}-\int_{\Delta_{\sigma^n\omega}} \psi_{\sigma^n\omega} d\mu_{\sigma^n\omega}\int_{\Delta_\omega} \vp_\omega d\leb_\omega\right| \le C_\omega\beta^n\|\vp_\omega\|_{\mF}\|\psi_{\sigma^n\omega}\|_\infty.$$
\end{theorem}

\begin{remark}
We do not obtain information on the random variable $C_\omega$ appearing before the
decay rate in the above theorem. This is because the base map of the random dynamical system under consideration is only ergodic and not necessarily mixing. See subsection \ref{sec:overview} below for an elaboration.
\end{remark}

\medskip
\subsection{Overview of the proofs}\label{sec:overview} We finish this section with a discussion on the techniques used to prove Theorem \ref{thm:main1}. We also comment on the assumptions and conclusions of the main theorems.  In contrast to random induced maps, which are piecewise expanding, one of the difficulties in dealing with random tower maps is that these admit a neutral behaviour similar to a discrete flow for points which have not yet returned to the base of the tower. As mentioned in the introduction, in this paper our technique is based on constructing suitable cones $\C^\omega(a,b,c)$ of functions defined on random towers.

We show that such cones contract, with respect to the Hilbert metric, under the action of appropriate random transfer operators, $P_\omega$, associated with the random tower maps. This idea is formalized in Proposition \ref{prop:key}, where we obtain that there exists $\kappa \in (0,1)$ and constants $a,b,c>0$ so that
\[
P^{k}_\omega\C^\omega(a,b,c)\subset\C^{\sigma^k\omega}(\kappa a,\kappa b, \kappa c)
\]
for $\mathbb P$-almost every $\omega\in\Omega$ and every $k\geqslant q_1(\omega)$. One cannot expect the latter to hold for all $k\geqslant k_0$ 
uniformly in $\omega$, the reason being that the proof of the invariance of the cone depends crucially on the fibered mixing property (Theorem~\ref{thm:main0}, in particular \eqref{eq:fibmix}) which is an asymptotic object.  At this point one uses the uniform exponential tail assumption\footnote{Condition (P5) is also used in Theorem \ref{thm:main0} to obtain the uniform, in $\omega$ and $x$, lower bound on the equivariant densities, see footnote 6 later in the text. The latter characteristic of the equivariant densities is needed in the cone condition that leads to the contraction of the random transfer operator in the Hilbert metric, see  Proposition \ref{prop:key}.} (P5).  Therefore, one proves that the transfer operator is a contraction with respect to Birkhoff's projective metric, 
(cf. Lemma~\ref{lem:mix})
In order to prove contraction of the random transfer operator we write
\[
P^n_\omega=P^{n-t_{s}}_{\sigma^{t_{s}}\omega}\circ P_{\sigma^{t_{s-1}}\omega}^{t_s-t_{s-1}}\circ\dots\circ P_{\sigma^{t_1}\omega}^{t_2-t_1}\circ P_\omega^{t_1},
\]
for some well chosen non-negative integer valued random variables $t_i=t_i(\omega)$ so that the cone field $\C^\omega(a,b,c)$
is preserved by iterations of the operators $P_\omega^{t_1}$ and $P^{t_{i+1}-t_i}_{\sigma^{t_i}\omega}$ (see Lemma~\ref{lem:erg}). The asymptotic positive density of such integers $t_i(\omega)$ is guaranteed by ergodicity and a selection lemma involving suitable powers of the shift map (see Lemma~\ref{lem:erg}). Thus, one concludes that the operator $P^n_\omega$ acts as a contraction on observables on the cone $\C^\omega(a,b,c)$,  observed for integer values $n\gg t_1(\omega)$.
This is ultimately related with the quenched exponential decay of correlations in Theorem~\ref{thm:main1}, where the random variable $C_\omega$ appears intrinsically.
\section{Applications}\label{sec:app}
In this section we apply Theorems \ref{thm:main0} and \ref{thm:main1} to specific random systems. First, for the sake of completeness we show the results of Theorems \ref{thm:main0} and \ref{thm:main1} can be passed from the random towers to the original random dynamical system \eqref{eq:randomor}. Here in the applications we assume that $X$ is a manifold with differentiable structure. Let $\pi_\omega: \Delta_\omega \to X$  defined by $\pi_\omega (x,\ell) = f_{\sigma^{-\ell}\omega}^{\ell} (x)$. Note that $\pi_{\sigma \omega} \circ F_\omega = f_\omega \circ\pi_\omega $. Then it is easy to see $\nu_\omega=(\pi_{\omega})_\star\mu_\omega$ is an equivariant family of measures for $\{f_\omega\}$ and its absolute continuity follows from the fact that $f_\omega$ are non-singular. Then lifting the observables $\phi_1 \in L^\infty(X)$ and $\phi_2 \in C^\eta(X)$ to observables $\bar{\varphi}_\omega = \phi_1\circ \pi_\omega$ and $\bar{\psi}_\omega = \phi_2\circ  \pi_\omega$ on the random towers, we get: 
\begin{align}\label{eq:corrpas}
\Big| \int_X (\phi_1 \circ f_\omega^n) \phi_2 d\nu_\omega & - \int_X \phi_1 d\nu_{\sigma^n \omega} \int_X \phi_2 d\nu_\omega \Big| = \\
&\Big| \int_{\Delta_{\omega}} (\bar{\varphi}_{\sigma^n \omega} \circ F_\omega^n) \bar{\psi}_\omega h_\omega d\lambda_{\omega} - \int_{\Delta_{\sigma^n\omega}} \bar \varphi_{\sigma^n \omega} d\mu_{\sigma^n \omega} \int_{\Delta_\omega} \bar \psi_\omega h_\omega d\lambda_\omega \Big|,
\end{align}
where $h_\omega=\frac{d\mu_\omega}{d\lambda_\omega}$. Moreover, by (P2) and (P3) it is easy to verify that $\bar{\varphi}_\omega \in {L}_\infty$ and $\bar \psi_\omega h_\omega \in \mathcal{F}$. Hence, using Theorem \ref{thm:main1} and \eqref{eq:corrpas} we obtain quenched exponential decay of correlations for the original random dynamical system. We now apply this to specific examples.
\subsection{Small random perturbations of one dimensional Lorenz maps with an ergodic driving system.}
In this example we obtain quenched exponential decay of correlations for small random perturbations of one dimensional Lorenz maps with an ergodic driving system.\\
Let $X=[-\frac12, \frac12]$ and $f:X\setminus\{0\} \to X$ be $C^{1}$  with a singularity at $0$ and one-sided limits 
\mbox{$f(0^+)<0$} and $f(0^-)>0$. Assume:
\begin{itemize}
    \item There are constants $\tilde C>0$ and $\ell>0$  such that $Df^n(x)> \tilde{C} e^{n \ell}$ for all $n\ge1$ whenever $x\notin \bigcup_{j=0}^{n-1}(f^{j})^{-1}(0)$;\\
    \item There exists $C>1$ and $0<\alpha<\frac12$ such that in a one sided neighborhood of $0$ 
\begin{align*}C^{-1}|x|^{\alpha-1}\le |Df(x)|\le C|x|^{\alpha-1}, 
\end{align*}
and $1/f$ is H\"older continuous on $[-\frac12, 0]$ and $[0,  \frac12]$;
\item $f$ is transitive.
\end{itemize}
We consider  $\mathcal U$  a small neighbourhood of $f$ in a suitable $C^{1+\alpha}$ topology (see \cite{La22} for precise definition) and $\PP$ be a compactly supported  Borel probability measure whose support $\Omega$ is contained in $\mathcal U$ and $f\in\Omega$. Let $\sigma: \Omega \to \Omega$ be an \emph{ergodic} automorphism of the probability space $(\Omega, \PP)$. We then form the skew product as in \eqref{eq:skew1}
and study the random orbits as defined in \eqref{eq:randomor}. Section 3 of \cite{La22} shows that the above random system admits a random tower that satisfies (P1)-(P5). Consequently, we apply the results of Theorems \ref{thm:main0} and \ref{thm:main1} to the random system defined in \eqref{eq:randomor} to obtain quenched exponential decay of correlations. 
\subsection{Random quadratic maps with an ergodic driving system.} 
In what what follows we consider random perturbations of the full quadratic map. Let $X=[0,1]$ and $f(x)=4x(1-x)$ on $X$. Recall that $\frac34$ is a fixed point of $f$ and it has a unique preimage $x_0\in[0, \frac12]$.   Let $f_\alpha:X\to X$ be a $C^3$ map with negative Schwarzian derivative,  such that $f_\alpha(x)=f(x)$ for all $x\in[0, x_0)\cup(\frac34, 1]$ and it has a unique critical point $c_\alpha\in(x_0, \frac34)$ of order $\alpha> 1$ such that $f_\alpha(c_\alpha)=1$. Further we assume that  $f_\alpha(x_0)=f_\alpha(\frac34)=\frac34$. 

Let $(\Omega, \PP,\sigma)$ be an ergodic automorphism. Assume that there is a mapping $\alpha:\Omega\to J\subset (1, +\infty)$ be a compact interval, so that $\alpha(\omega)$ defines a map $f_\omega=f_{\alpha(\omega)}:X\to X$ with the above properties.  We then form the skew product
We then form the skew product as in \eqref{eq:skew1}
and study the random orbits as defined in \eqref{eq:randomor}. 
Next we show that the random system \eqref{eq:randomor} satisfies (P1)-(P5). Consider the deterministic sequence: $x_n=(f|_{[0,x_0)})^{-1}x_{n-1}$. Notice that $x_n\to 0$ as $n\to\infty$ and since  ${f}'(x)>2$  whenever $x\in[0,\frac14]$, $x_n$ converges to zero exponentially fast. 
Let $I_n=[x_{n}, x_{n-1})$ for $n\ge 1$ and $I_0=[x_0, \frac34]$.  Then $I_n$ has the following properties:
\begin{itemize}
\item $f^n_\omega(I_n)=[x_0, \frac34)$. 
\item There exists $C>0$ such that $|I_n|\le \frac{C}{2^{n+1}}$. 
\end{itemize}
Next we define preimages of $x_n$ under the right branch of $f$. Let $y_n=(f|_{(\frac34,1]})^{-1}x_{n-1}$ for $n\ge 1$. For convenience  we also let $y_0=\frac34$.  Again we can find $C>0$
\begin{equation}\label{eq:yn}
|y_{n}-y_{n-1}|\le \frac{C}{2^n} \text{ and } f((y_{n}, y_{n-1}])=I_{n-1}.
\end{equation} 

Finally, using the $y_n$'s and $f_\omega$ define a random partition of $I_0$ and return times as follows. Recall that $f_\omega$ has a unique critical point $c_{\alpha(\omega)}\in(x_0,\frac34)$. Thus for every $\omega\in\Omega$ let 
\[
z_n^-(\omega)=(f_\omega|_{(x_0, c_{\alpha(\omega)})})^{-1}y_n, \quad 
z_n^+(\omega)=(f_\omega|_{(c_{\alpha(\omega)}, \frac34)})^{-1}y_n
\]
Since $\alpha(\omega)$ is uniformly bounded, we there exists $\eta\in(0,1)$ and a constant $C>0$ such that 
\begin{equation}\label{eq:zn}
|z_n^-(\omega)-z_{n-1}^-(\omega)|\le \frac{C}{2^{\eta n}} \text{ and }
|z_{n-1}^+(\omega)-z_{n}^+(\omega)|\le \frac{C}{2^{\eta n}}.
\end{equation}

From the definition it is evident that $z_n^-(\omega)$ is increasing, $z_n^+(\omega)$ is decreasing and both converge to $c_{\alpha(\omega)}$  Thus, $\{z_n^-(\omega)\}\cup \{z_n^+(\omega)\}$ defines a partition on $[x_0, \frac 34]$, which allows us to define the desired return time. Let $R_\omega(x)=2$ for every $x\in [x_0, z_1^-(\omega))\cup (z_1^+(\omega),\frac34]$. For all $n\ge 1$ set 
$$R_\omega(x)=n \quad\text{for}\quad x\in [z_n^-(\omega), z_{n+1}^-(\omega))\cup (z_{n+1}^+(\omega), z_n^+(\omega)].$$ 
By construction, $f_\omega^{R_\omega}$ is full branch and for every $n\ge 2$ there are two intervals returning at time $n$. Thus, the aperiodicity assumption is satisfied automatically. Tail estimates follow from \eqref{eq:zn}. Distortion bounds follow from negative Schwarzian assumption. Finally, weak expansion assumption follows from bounded distortion and tail estimates.  Consequently, we apply the results of Theorems \ref{thm:main0} and \ref{thm:main1} to the random system defined in \eqref{eq:randomor} to obtain quenched exponential decay of correlations. 

\subsection{Small random perturbations of Axiom A with an ergodic driving system.}
In this example we apply Theorems \ref{thm:main0} and \ref{thm:main1} together with Theorem 1.6 of \cite{ABR22} to obtain quenched exponential decay of correlations for small random perturbations of Axiom A diffeomorphisms with an ergodic driving system. More precisely, let $X$ be a smooth compact Riemannian manifold of finite dimension. Denote by $\diff^{1+}(X)$ the set of $C^{1}$ diffeomorphisms whose derivative is H\"older continuous endowed with the $C^1$ topology. Let $f\in\diff^{1+}(X)$ be a topologically mixing uniformly hyperbolic Axiom A diffeomorphism \cite{BR75}. We consider  $\mathcal U$  a small neighbourhood of $f$ in the $C^1$ topology and $\PP$ a compactly supported  Borel probability measure whose support $\Omega$ is contained in $\mathcal U$ and $f\in\Omega$. Let $\sigma: \Omega \to \Omega$ be an \emph{ergodic} automorphism of the probability space $(\Omega, \PP)$. We then form the skew product as in \eqref{eq:skew1}
and study the random orbits as defined in \eqref{eq:randomor}. Theorem 1.6  and subsection 3.5 of \cite{ABR22} show that the above random system admits  a random hyperbolic tower, as defined in \cite{ABR22}, with exponential tails. Then, from this random hyperbolic tower, we obtain a quotient random tower that satisfies (P1)-(P5). Finally, we apply the results of Theorems \ref{thm:main0} and \ref{thm:main1} to the random system defined in \eqref{eq:randomor} to obtain quenched exponential decay of correlations for H\"older observables on $X$. 
\section{Existence, uniqueness and fibrewise mixing}\label{sec:proof1}
\begin{proof}[Proof of Theorem \ref{thm:main0}]
Recall the set $\mathbb \Delta$ and the skew-product $F:\mathbb \Delta\to\mathbb \Delta$ defined in \eqref{eq:skew}. Set $P=\mathbb P\times\leb_\omega$ and
define $\mu^0=P|_{\mathbb \Delta_0}$ where 
$\mathbb \Delta_0=\{(\omega, x)\mid \omega\in\Omega, x\in\Delta_{\omega,0}\}$. For $n\ge 1$, let
$\mu_n=\frac1n\sum_{j=0}^{n-1}F^n_\ast \mu^0$. It is known that $\mu_n$ admits accumulation points in the weak$^\ast$ topology, but  the accumulation points need not be probability measures. We let $\mu$ be a weak$^*$ accumulation point of $\mu_n$ and below we show that $\mu$ is indeed a probability measure that is absolutely continuous with respect to $P$.

We first prove absolute continuity. For any $j\ge 1$ set
\begin{equation*}
\mathcal P_\omega^{j}=\vee_{i=0}^{j-1}F_\omega^{-i}\mP_{\sigma^i\omega}   \quad\text{and} \quad 
\mathcal A_\omega^{j}=\{A\in \mathcal P_{\omega}^{j} \mid F^j_\omega A=\Delta_{\sigma^{j}\omega, 0}\}.
\end{equation*}
For $A_\omega\in\mP_{\omega}^{j}|\Delta_{\omega, 0}$, $A=\{(\omega, x)\mid \omega\in\Omega, x\in A_\omega\}$. Since $\PP$ is $\sigma$-invariant we have  
$$
\phi_{j, A}(\sigma^j\omega, x) =\frac{dF^j_\ast(\mu^0_{A})}{dP}=J(F_\omega^j)^{-1}(x),\quad\mu^0_{A}(B)=\mu^0(A\cap B).
$$ 
Clearly, $\phi_{j, A}$ is a density on $\mathbb\Delta$. 
Below we consider two cases depending on $A$. First, notice that if  $A_\omega\in\mathcal A_{\omega}^{j}$ then $F^j_{\omega}:A_\omega\to \Delta_{\sigma^{j}\omega,0}$ is a bijection.
For $x, y \in \Delta_{\sigma^{j}\omega,0}$, let $x', y' \in A_\omega$ be such that $F^j_{\omega}(x')=x$, and  $F^j_{\omega}(y')=y$. Then there exists $i$ so that $F^j_{\omega}= (F^{R_\omega}_{\omega})^i$. The bounded distortion condition (P2) implies that 

\begin{equation}\begin{split}\label{eq:density_reg}
\left|\log \frac{\phi_{j, A}(\sigma^{j}\omega, y)}{\phi_{j, A}(\sigma^{j}\omega, x)}\right|=
\left|\log \frac{JF^j_{\omega}(x')}{JF^j_{\omega}(y')}  \right|
\le \sum_{\ell=0}^{i-1} D\gamma^{s(x, y)+(i-\ell)-1} \le \frac{D}{1-\gamma} \gamma^{s(x, y)}.
\end{split}
\end{equation}
Therefore, there exists $D'>0$ independent of $j$, $A$ and $\omega$
such that 
$$
 \phi_{j, A}(\sigma^j\omega, y) \le D'\phi_{j, A}(\sigma^j\omega,x), \quad x, y\in  \Delta_{\sigma^j\omega, 0}.
$$
Integrating both sides  of the latter inequality over $\Delta_{\sigma^j\omega, 0}$ with respect to $x$  implies
\begin{equation}\label{phiomegajA}
\phi_{j, A}(\sigma^j\omega, y) \le D' \frac{\leb_{\sigma^j\omega}(A_{\sigma^j\omega})}{\leb_{\sigma^j\omega}(\Delta_{\sigma^j\omega})}= D'\leb_{\sigma^j\omega}(A_{\sigma^j\omega}).
\end{equation}
On the other hand, if  $A_\omega\in \mP_{\omega}^{(j)}|\Delta_{\omega, 0}$ such that $F^j_{\omega}(A_{\omega})\subset \Delta_{{\sigma^j\omega}, \ell}$ for $\ell>0$ then  $\phi_{j, A}(\sigma^j\omega, y)(x)=\phi_{j-\ell, A}^{\sigma^{j-\ell}\omega} (\sigma^{j-\ell}\omega, F_{\sigma^{-\ell}\omega}^{-\ell}(x))$ and therefore, by \eqref{phiomegajA} it is bounded. Hence we get $d(F_\ast^n\mu_0)/dP<D'$ for every $n\ge 1$. This implies absolute continuity of $\mu$. Now we are in position to show that $\mu$ is indeed a probability measure.
We have
\begin{equation}\label{low:phiomegajA}
\phi_{j, A}(\sigma^j\omega, y) \ge\frac 1{D'}\leb_{\sigma^j\omega}(A_{\sigma^j\omega}).
\end{equation}
By \eqref{low:phiomegajA} for every $n\ge 1$ and for every continuous $\varphi:{\mathbb \Delta}\to \mathbb R$ we have 
\[
\int_{\mathbb \Delta} \vp d\mu_n\ge \frac1n\sum_{j=0}^{n-1}\int_{\mathbb \Delta} \vp\frac{dF^n_\ast\mu_0}{d\mu_0}d\mu_0\ge \frac1{D'}\int_{\mathbb \Delta}\vp d\mu_0=\frac1{D'}\int_{\mathbb \Delta_0}\vp dP.
\]
Thus, $\mu_n$ cannot converge to $0$ in the weak$^*$ topology. Hence, by this and the absolute continuity of $\mu$, we get that $\mu$ is a finite positive measure that can be normalised. We also call the normalised $F$-invariant probability measure $\mu$.

By Rohklin's theorem one can disintegrate $\mu$ to obtain measurable family of equivariant measures $\{\tilde\mu_\omega\}$, i.e. $\omega\mapsto \tilde\mu_\omega$ is measurable and $(F_\omega)_\ast\tilde\mu_\omega=\tilde\mu_{\sigma\omega}$. The absolute continuity of $\mu$ with respect to $P$ readily implies the absolute continuity of $\tilde\mu_\omega$ with respect to $\lambda_\omega$ for $\PP$-almost every $\omega\in\Omega$. 

Now we use results from the literature to show that $\tilde\mu$ is unique and $1/C<d\tilde\mu_\omega/d\lambda_\omega<C$, $C>0$. In \cite[Theorem 2.10 ]{AV13} the existence of fibrewise measures $\{\mu_{\omega}\}$ satisfying $(F_\omega)_\ast\mu_\omega=\mu_{\sigma\omega}$ is proved. Moreover, in \cite[Theorem 2.10 ]{AV13} it is proved that\footnote{The uniform summability assumption, $\sum_{n\ge0} \lambda_{\sigma^{-n}\omega}\{R_{\sigma^{-n}\omega} >n\}\le C$, of \cite{AV13} follows from the uniform decay (P5) in our paper.} $1/C<d\mu_\omega/d\leb_\omega<C$, $C>0$. While the uniqueness of $\{\mu_{\omega}\}$ is proved in \cite{H22}. Therefore, $\tilde\mu_\omega=\mu_\omega$ for $\PP$-almost every $\omega\in\Omega$ and by our work above $\omega\to\mu_\omega$ is measurable. Now that $\omega\to\mu_\omega$ is measurable, the fibrewise mixing property follows verbatim from [\cite{BBMD}, Lemma 4.2]. This finishes the proof of Theorem~\ref{thm:main0}.
\end{proof}

\section{exponential operational correlations}\label{sec:proof2}
We prove Theorem \ref{thm:main1} in a series of lemmas and propositions using transfer operators on random towers and appropriate Birkhoff cones \cite{Birk1,Birk2}. 

\subsection{Transfer operators for random towers and Lasota-Yorke inequalities}
 Let $v_\omega:\Lambda_\omega\to\mathbb R$ be given by: for $x\in \Delta_{\omega,\ell}$, $$v_\omega(x)=e^{\ell\theta'},$$ 
 where $\theta'\in(0,\theta)$ and $\theta$ is as in (P5).  For any measurable set $A$ define $m_\omega(A)=\int_Av_\omega(x) d{\leb_\omega}$.
The transfer operator associated with the random tower map is defined by
\begin{equation}\label{eq:transnormal}
\mL_\omega\psi_\omega(x)=\sum_{F_\omega(y)=x}JF_\omega^{-1}(y)\psi_\omega(y).
\end{equation}
 It is easy to check that if there exists $h_\omega\ge 0$, with $\int h_\omega d\leb_\omega=1$ and $\mL_\omega h_\omega=h_{\sigma\omega}$, then  $\{\mu_\omega=h_\omega\cdot \leb_\omega\}$ is an absolutely continuous equivariant family of probability measures.
In our work we also use another transfer operator to deduce results about $\mL_\omega$, namely 
\begin{equation*}
    P_\omega\psi_\omega:=v^{-1}_{\sigma\omega}\mL_\omega(v_\omega\psi_\omega),
\end{equation*}
where the function $v_\omega$ is defined above.  Moreover,
\begin{equation*}
    P^n_\omega\psi_\omega:=v^{-1}_{\sigma^n\omega}\mL_\omega^n(v_\omega\psi_\omega),
\end{equation*}
and $P_\omega$ satisfies the following duality: for $\psi\in \mF$ and $\varphi\in L^\infty(\Delta)$
\begin{equation}\label{eq:duality}
\int \varphi_{\sigma \omega}\circ F_\omega\cdot\psi_\omega dm_\omega=\int P_\omega\psi_\omega\cdot\varphi_{\sigma\omega} dm_{\sigma\omega}.    
\end{equation}
Consequently, $m_\omega$ is a conformal measure with respect to $P_\omega$, i.e. $P_\omega^*m_\omega=m_{\sigma\omega}$. Notice also, by Theorem \ref{thm:main0}, $\tilde h_\omega:=h_{\omega}/v_{\omega}$ satisfies $P_\omega\tilde h_\omega=\tilde h_{\sigma\omega}$.

For any $n\in\NN$ let $x_i\in \Delta_{\sigma^n\omega}$, and $y_i\in \Delta_{\omega, \ell}^j$ be such that $F^n_\omega(y_i)=x_i$, for $i=1,2$. Then we have $s_\omega(y_1,y_2)=s_{\sigma^n\omega}(x_1, x_2)+n$.
Also, by \eqref{bddd} we have 
\begin{equation}\label{bdddn}
\left|\frac{JF_\omega^n(y_1)}{JF_\omega^n(y_2)}-1\right|\le D_Fd_{\sigma^n\omega}(x_1, x_2),
\end{equation}
for some $D_F$ independent of $n$.
In what follows we will avoid enumerating constants and keep denoting by $C>0$ some large constant, independent of $\omega$.

\begin{lemma}
There exists $C>0$ such that for any $n\in\NN$ we have
\begin{equation}\label{eq:Uni}
    |P^n_\omega{\bf 1}|_{\infty}\le C.
\end{equation}

\end{lemma}
\begin{proof}
From \eqref{bdddn} there exists $D_F>0$ such that for any $\Delta_{\omega, \ell}^\ast\subset\Delta_{\omega, \ell}$ with $F_\omega^n(\Delta_{\omega, \ell}^\ast)=\Delta_{\sigma^n\omega, 0}$ and $x\in\Delta_{\sigma^n\omega, 0}$ we have 
\[
\frac{1}{D_F}\frac{\leb_\omega(\Delta_{\omega, \ell}^\ast)}{\leb_\omega(\Delta_{\sigma\omega, 0})}\le \frac{1}{JF^n_\omega(x)}\le D_F\frac{\leb_\omega(\Delta_{\omega, \ell}^\ast)}{\leb_\omega(\Delta_{\sigma\omega, 0})}.
\]
Thus, using (P5), for any $x\in \Delta_{\sigma^n\omega, 0}$ we have 
\begin{equation}\label{eq:l0}
|P^n_\omega{\bf 1}(x)|  \le D_F \sum_{\ell\ge 0}\frac{\leb_\omega(\Delta_{\omega, \ell})}{\leb_\omega(\Delta_{\sigma\omega, 0})} e^{\ell\theta'}\le C'\sum_{\ell\ge 0}e^{(\theta'-\theta)\ell} \le C.
\end{equation}
For $x\in \Delta_{\sigma^n\omega, \ell}$ with $1\le \ell \le n$ we have 
\[
|P^n_\omega{\bf 1}(x)|=e^{-\ell\theta'}|P^n_\omega{\bf 1}(F^{-\ell}_\omega x)|< C, 
\]
where in the last inequality we used \eqref{eq:l0}. 
Finally, for $x\in \Delta_{\sigma^n\omega, \ell}$ with $\ell\ge n$ we have 
\[
|P^n_\omega{\bf 1}(x)|=e^{-\ell\theta'} e^{(\ell-n)\theta'}= e^{-n\theta'}\le 1.
\]
This proves the lemma.
\end{proof}
Recall that $\mP_\omega$ is the partition of the tower $\Delta_\omega$ and define
\[
\mP^{k}_\omega=\bigvee_{j=0}^{k-1}F^{-j}_\omega\mP_{\sigma^{j}\omega}.
\]
We denote by $C_{k, \omega}(x)$ the element of $\mP_\omega^k$ containing $x$. Let  $\rho_\omega:\Delta_\omega\to \mathbb R$ be defined  as $\rho_\omega(x)=\gamma$, if $x$ is in the base, and  $\rho_\omega(x)=1$ otherwise. For $k\in \mathbb N$ let $\rho^{(k)}_\omega(x)=\prod_{i=0}^k\rho_{\sigma^i\omega}(F^i_\omega x)$. Recall that if $x$ and $y$ are in $\Delta_{\omega,0}$ and $F^{k}(x')=x$ then there is a unique $y'\in C_{k, \omega}(x')$, the latter is the $k^{\text{th}}$ cylinder containing $x'$, such that $F^{k}(y')=y$. Let 
$\zeta\in (0, 1)$  be a small number to be chosen below. Let 
 \[
 \mathcal G_\omega^c=\Bigg( \bigcup_{\hat R^j_{\omega, \ell}>\lfloor\zeta n\rfloor}\Delta_{\omega, \ell}^j \Bigg) \bigcup \left\{x\mid R_{\lfloor\zeta n\rfloor}(\omega, x)> n\right\}
 \]
 and $ \mathcal G_\omega$ denote its complement.
For each $k\ge 1$ write 
 $$ \bar\mP_\omega:=\mP^{(k)}_\omega|_{\mathcal G_\omega};\quad \bar\mP^c_\omega:=\{\mathcal G_\omega^c\} \quad\text{and}\quad \hat\mP_\omega:=\bar\mP_\omega \cup \bar\mP^c_\omega,$$
where we omit the dependence of $\bar\mP_\omega$ and $\hat\mP_\omega$ on $k$ and the dependence of 
$\mathcal G_\omega^c$ on $n$ for notational simplicity.
In rough terms, $\bar\mP^c_\omega$ consists of the set of points in the tail of the inducing time and the partition
$\bar\mP_\omega$ is formed by those elements of $\mP^{k}_\omega$ in the set of points with smaller inducing times.

\begin{remark}
In \cite{H22} Hafouta studied probabilistic limit theorems for non-uniformly hyperbolic systems in an iid setting, using Birkhoff cones techniques on random towers. The results of \cite{H22} use the rates of correlations decay of \cite{ABR22}, which are in an iid setting. In particular, the results of \cite{ABR22} were used by  \cite{H22} to prove a stronger version of Lemma \ref{lem:mix} below with uniform bounds independent of $\omega$ (see Lemma 4.2.2 of \cite{H22}, in particular equation (4.17) of \cite{H22}). However, in our work, since we are after the correlations decay rates to start with, and in a non-iid setting, Lemma \ref{lem:mix} below is the best that one can hope for in terms of fibrewise mixing. In fact, the random variables that appear in Lemma \ref{lem:mix} below are the main source of difficulty in proving the correlation decay rate in our non-iid setting, and they eventually lead to the random variable $C_\omega$ that appears in Theorem \ref{thm:main1}.
\end{remark}

\begin{lemma}\label{lem:mix}
There exist constants $\alpha<1< \alpha'$  and a random variable $q_0:\Omega\to \NN$ such that 
\begin{equation}\label{eq:mix}
\alpha<\frac{m_\omega(A_\omega\cap F^{-k}_\omega(A'_{\sigma^k\omega}))}{m_\omega(A_\omega)\mu_{\sigma^k\omega}(A'_{\sigma^k\omega})}<\alpha'.
\end{equation}
for all $k\ge q_0(\omega)$, $A_\omega\in \hat\mP_\omega$ and $A_{\sigma^k\omega}'\in \hat\mP_{\sigma^k\omega}$.
\end{lemma}

\begin{proof}
For each $k\in \NN$, since both $\hat\mP_\omega$ and $\hat\mP_{\sigma^k \omega}$ are finite partitions, the function 
\[
\rho (A_\omega, A_{\sigma^k\omega}')=\frac{\mu_\omega(A_\omega\cap F^{-k}_\omega(A'_{\sigma^k\omega}))}{\mu_\omega(A_\omega)\mu_{\sigma^k\omega}(A'_{\sigma^k\omega})}.
\]
has finitely many values. Since $\{\mu_\omega\}$ is fibrewise mixing, we have $\lim_{k\to\infty}\rho (A_\omega, A_{\sigma^k\omega}') =1$, and hence  for any $\eps>0$ there exists $q_0(\eps, \omega)\in\NN$ such that  
\[
1-\eps <\rho (A_\omega, A_{\sigma^k\omega}') <1+\eps,
\]
for  all $k\ge q_0(\eps, \omega)$,  $A_\omega\in \hat\mP_\omega$ and $A_{\sigma^k\omega}'\in \hat\mP_{\sigma^k\omega}$. Since  $1/C<d\mu_\omega/d\leb_\omega<C$ for $\PP$-almost all $\omega\in\Omega$, choosing $\eps=1/2$ we have 
\[
\frac{1}{2C^2}< \frac{\leb_\omega(A_\omega\cap F^{-k}_\omega(A'_{\sigma^k\omega}))} {\leb_\omega(A_\omega)\mu_{\sigma^k\omega}(A'_{\sigma^k\omega})}<\frac{3C^2}{2}.
\]
Choosing $\alpha=\frac{1}{2C^2}$, $\alpha'=\frac{3C^2}{2}$ and using the definition of $m_\omega$ finishes the proof. 
\end{proof} 

\begin{lemma}\label{lem:rhon}
There exist $\zeta>0$ and $\bar\theta>0$ such that for every $n\in\NN$, every $x\in \mathcal G_\omega$ and $\PP$-almost every $\omega \in \Omega$ 
we have 
$$
m_\omega(\mathcal G_\omega^c)\le Ce^{-\bar\theta n}\quad\text{and} \quad  
	\rho^{(n)}_{\omega}(x)\le \gamma^{\lfloor\zeta n\rfloor}.
$$
\end{lemma}

\begin{proof}
Recall that, by definition, if  $x\in \mathcal G_\omega$ then the orbit of $x$ has at least $\lfloor\zeta n\rfloor$ returns 
to the base $\Delta_{\omega,0}$ of the tower before time $n$. Hence, $\rho^{(n)}_\omega(x)\le \gamma^{\lfloor\zeta n\rfloor}$. It remains to estimate the measure of $\mathcal G_\omega^c$.  Bounded distortion \eqref{bdddn} and the tail estimate (P5) implies that for every $k_1, k_2, \dots, k_{\lfloor\zeta n\rfloor-1}\in \NN$ such that $\sum_{i=1}^{\lfloor\zeta n\rfloor-1}k_i<n$ we have 
 
\begin{align*}
&\leb_\omega\{R_1(\omega, x)=k_1\}\times\\
&\prod_{i=2}^{\lfloor\zeta n\rfloor-1} \leb_\omega\{R_i(\omega, x) - R_{i-1}(\omega, x)=k_i\mid R_j-R_{j-1}=k_j, 2\le j\le i-1\} \\
&\times \leb_\omega\{R_{\lfloor\zeta n\rfloor}(\omega, x) - R_{\lfloor\zeta n\rfloor-1}(\omega, x)> n-\sum_{i}k_i\}\\
&\le (CD')^{\lfloor\zeta n\rfloor}\prod_{i=1}^{\lfloor\zeta n\rfloor-1}e^{(k_i-1)\theta}\times e^{(n-\sum_ik_i)\theta}
\le e^{-n\theta+ \lfloor\zeta n\rfloor(\log(CD')+\theta)}
\end{align*}
Moreover, using  Lemma 3.4 from \cite{BLvS}: for each small $\zeta\in (0, 1)$ there exists $\hat\zeta>0$ such that $\hat\zeta\to 0$ as $\zeta\to 0$ and  
\[
\text{card}\{(k_1, \dots, k_s)\in \NN^s\mid k_1+\dots+ k_s=n, s\le \lfloor\zeta n\rfloor\}\le e^{\hat\zeta n}, 
\]
for every $n\in\NN$.
Thus, we obtain 
\begin{align*}
\leb_\omega(\mathcal G_\omega^c)
	& \le \text{card}\{(k_1, \dots, k_s)\in \NN^s\mid k_1+\dots+ k_s=n, s\le \lfloor\zeta n\rfloor\}e^{-n\theta+ \lfloor\zeta n\rfloor(\log(CD')+\theta)}\\
	& \le e^{-n\theta+ \lfloor\zeta n\rfloor(\log(CD')+\theta)+\hat\zeta n}=e^{-\hat\theta n},
\end{align*}
for some $\hat\theta>0$. This is possible because, for sufficiently small $\zeta>0$ we have $-\theta+ \lfloor\zeta n\rfloor(\log(CD')+\theta)/n+\hat\zeta =-\hat \theta<0$. Finally, using the latter and (P5), 
\begin{align*}
m_\omega(\mathcal G_\omega^c)
	& =\sum_{\ell=0}^{\infty}\leb_\omega(\mathcal G_\omega^c\cap\Delta_{\omega,\ell})e^{\ell\theta'}\\
	& \le\sum_{\ell=0}^{\lfloor n/2\rfloor}\leb_\omega(\mathcal G_\omega^c)e^{\ell\theta'}+\sum_{\ell=\lfloor n/2\rfloor+1}^{\infty}\leb_\omega(\Delta_{\omega,\ell})e^{\ell\theta'}
	\le C e^{-\bar\theta n},
\end{align*}
where $\bar\theta=\min\{\frac{\hat\theta}{2},\frac{\theta-\theta'}{2}\}>0$. This completes the proof of the lemma.
\end{proof}

 Let 
 \begin{equation}\label{eq:dconst}
 \D_1:=\text{diam}\bar\mP_\omega\quad\text{ and } \quad\mathcal D_2:=m_\omega(\mathcal G_\omega^c).
 \end{equation}
 
 Notice that $\D_1$ is smaller than $\gamma^{\lfloor\zeta n\rfloor}$ and by Lemma \ref{lem:rhon}, $\mathcal D_2\le C e^{- \bar\theta n}$. Now we are ready to obtain Lasota-Yorke type estimates.
 
\begin{lemma}\label{le:estimateaux}
There exists $C>0$ and for each $\eta>0$ there exists $k_0(\eta)\in\NN$ such that for all $k>k_0$ and $\psi=(\psi_\omega)_\omega \in\mF$ we have
\begin{itemize}
\item[1.] if $\ell\ge k$ and $x,y\in\Delta_{\sigma^k\omega,\ell}$ then 
$$
|P_\omega^k\psi_\omega(x)-P_\omega^k\psi_\omega(y)|=e^{-\theta'k}|\psi_\omega|_h \,d_{\sigma^k\omega}(x, y);
$$
\item[2.] if $\ell< k$ and $x,y\in\Delta_{\sigma^k\omega,\ell}$ then
$$|P_\omega^k\psi_\omega(x) - P_\omega^k\psi_\omega(x)|\le 
d_{\sigma^k\omega}(x,y)e^{-\theta'\ell}\left[(C\gamma^{\zeta k}+\eta D_F)|\psi_\omega|_h+D_FC|\psi_\omega|_{\infty}\right].
$$
\end{itemize}
\end{lemma}

\begin{proof} For  any $x, y\in \Delta_{\sigma^k\omega, \ell}$ and $k\le \ell$ we let  $x'=(F^{k}_\omega)^{-1}(x)$, $y'=(F^{k}_\omega)^{-1}(y)$. Then, both of the preimages are in the base at fiber $\omega$. We have 
\begin{equation*}
|P_\omega^k\psi_\omega(x)-P_\omega^k\psi_\omega(y)|= |e^{-\theta'k}\psi_\omega(F^{-k}_\omega x)-e^{-\theta'k}\psi_\omega(F^{-k}_\omega y)|\le e^{-\theta'k}\, d_{\sigma^k\omega}(x, y)\, |\psi_\omega|_h.
\end{equation*}
For $x, y\in \Delta_{\sigma^k\omega, \ell}$ and $\ell <k$ notice that both $x_0=F^{-\ell}_{\sigma^{k}\omega}x$ and $y_0= F^{-\ell}_{\sigma^{k}\omega}y$ belong to the same ground floor $\Delta_{\sigma^{k-\ell}\omega,0}$ and we get 
\begin{equation}\label{eq:k-ell}
|P_\omega^k\psi_\omega(x)- P_\omega^k\psi_\omega(y)|= e^{-\theta'\ell}|P_\omega^{k-\ell}\psi_\omega(x_0)- P_\omega^{k-\ell}\psi_\omega(y_0)|.
\end{equation}
Recall that if $F^{k-\ell}_\omega(x')=x_0$ then  every cylinder $C_{k-\ell}(x')$ containing $x'$ contains a unique pre-image $y'$ of $y_0$. Therefore,
\begin{equation}\label{eq:k-ell0}
\begin{split}
|P_\omega^{k-\ell}\psi_{\omega}(x_0)-P_\omega^{k-\ell}\psi_\omega(y_0)|
&=\left|\sum_{C_{k-\ell}}\frac{v_{\sigma^{k-\ell}\omega}(x')}{JF^{k-\ell}_\omega(x')}\psi_{\sigma^{k-\ell}\omega}(x')-\frac{v_{\sigma^{k-\ell}\omega}(y')}{JF^{k-\ell}_\omega(y')}\psi_{\sigma^{k-\ell}\omega}(y')\right|\\
&\le\sum_{C_{k-\ell}}\frac{v_{\sigma^{k-\ell}\omega}(x')}{JF^{k-\ell}_\omega(x')}
|\psi_{\sigma^{k-\ell}\omega}(x')-\psi_{\sigma^{k-\ell}\omega}(y')|\\
&+\sum_{C_{k-\ell}}|\psi_{\sigma^{k-\ell}\omega}(y')|
\left|\frac{v_{\sigma^{k-\ell}\omega}(x')}{JF^{k-\ell}_\omega(x')}-\frac{v_{\sigma^{k-\ell}\omega}(y')}{JF^{k-\ell}_\omega(y')}\right|,
\end{split} 
\end{equation}
where the previous sums are taken over all $(k-\ell)$-cylinders $C_{k-\ell}$, which contain paired preimages $x'$ and $y'$ (depending on $C_{k-\ell}$) for $x$ and $y$, respectively.
We estimate the second summand in \eqref{eq:k-ell0} as follows:
\begin{equation}\label{eq:k-ell1}
\begin{split}
&\sum_{C}\; |\psi_{\sigma^{k-\ell}\omega}(y')|\frac{v_{\sigma^{k-\ell}\omega}(x')}{JF^{k-\ell}_\omega(x')}
\left|\frac{JF^{k-\ell}_\omega(x')}{JF^{k-\ell}_\omega(y')}-1\right|\\ 
&\le D_F \; d_{\sigma^k\omega} (x,y)\sum_{C_{k-\ell}}|\psi_{\sigma^{k-\ell}\omega}(y')|\frac{v_{\sigma^{k-\ell}\omega}(x')}{JF^{k-\ell}_\omega(x')}\\
&\le D_F \; d_{\sigma^k\omega}(x,y)|\psi_{\sigma^{k-\ell}\omega}|_\infty|P_{\sigma^{k-\ell} \omega}^{k-\ell}{\bf 1}|_\infty.
\end{split}
\end{equation}
It remains to estimate
\begin{equation}\label{eq:k-ell2}
    \begin{split}
 \sum_{C_{k-\ell}} & \frac{v_{\sigma^{k-\ell}\omega}(x')}{JF^{k-\ell}_\omega(x')}|\psi_{\sigma^{k-\ell}\omega}(x')-\psi_{\sigma^{k-\ell}\omega}(y')| \\
 & \le 
 |\psi_{\sigma^{k-\ell}\omega}|_h
 \sum_{C_{k-\ell}}\frac{v_{\sigma^{k-\ell}\omega}(x')}{JF^{k-\ell}_\omega(x')} \; d_{\sigma^{k-\ell}\omega}(x' y') \\
 &\le |\psi_{\sigma^{k-\ell}\omega}|_h \; d_{\sigma^k\omega}(x, y)\left(
 \sum_{\substack{F^k(x')=x,\\x'\in \mathcal G_\omega}} 
 \frac{v_{\sigma^{k-\ell}\omega}(x')}{JF^{k-\ell}_\omega(x')}\rho(x')
 +\sum_{\substack{F^k(x')=x,\\x'\in \mathcal G_\omega^c}} \frac{v_{\sigma^{k-\ell}\omega}(x')}{JF^{k-\ell}_\omega(x')}\rho(x')
 \right)\\
& \le |\psi_{\sigma^{k-\ell}\omega}|_h \; d_{\sigma^k\omega}(x, y)
 \left( \gamma^{\lfloor\zeta k\rfloor}\cdot|P^{k-\ell}{\bf 1}|_\infty+ D_F\cdot m(\mathcal G_\omega^c)\right),
  \end{split}
 \end{equation}
where for $x\in \mathcal G_\omega$ we used the fact that $\rho^{(k)}(x)\le\gamma^{\lfloor\zeta k\rfloor}$ while for $x\in\mathcal G_\omega^c$ we used $\frac{1}{JF^{k-\ell}_{\omega}(x)}\le D_F \leb_\omega(C_{k-\ell}(x))$. Finally, we use \eqref{eq:k-ell0}, \eqref{eq:k-ell1} and \eqref{eq:k-ell2} in \eqref{eq:k-ell} to obtain
$$|P_\omega^k\psi(x) - P_\omega^k\psi(y)|\le 
d_{\sigma^k\omega}(x, y) \, e^{-\theta'\ell}\left[(C\gamma^{\zeta k}+\eta D_F)|\psi_\omega|_h+D_FC|\psi_\omega|_{\infty}\right],
$$
which proves the lemma.
\end{proof}

\begin{corollary}\label{cor:LY}
For each $\varepsilon>0$ there exist $k_0, N\in\NN$ such that for all $k>k_0$, $\psi=(\psi_\omega)_\omega \in\mF$ 
and $x,y \in\Delta_{\sigma^k\omega,\ell}$ we have
\begin{equation}\label{eq:LYC}
 |P_\omega^k\psi(x) - P_\omega^k\psi(y)|\le
 \begin{cases}
 \varepsilon |\psi_\omega|_h \, d_{\sigma^k\omega}(x,y) \text{ for } \ell\ge k;\\
 \varepsilon \|\psi_\omega\| \, d_{\sigma^k\omega}(x,y) \text{ for } N\le 2\ell< k;\\
 \left(\varepsilon|\psi_\omega|_h +D_FC|\psi_\omega|_\infty\right) \, d_{\sigma^k\omega}(x,y) \text{ for }  2\ell< N.
 \end{cases}
\end{equation}
 \end{corollary}
\begin{proof}
Given $\varepsilon>0$ we choose $N\in\NN$ such that $e^{-\theta'N/2}\max\{1, C+D_F, \eta D_F\}<\varepsilon$. Then the first two inequalities follow from Lemma~\ref{le:estimateaux} 
for all $k>N$. For the third inequality we choose $\eta>0$ and $k_0\in \NN$ such that $C\gamma^{\zeta k_0}+\eta D_F<\varepsilon$.
\end{proof}

\subsection{Birkhoff cones and the Hilbert metric}  
 We define the following cone
\begin{equation}
\begin{split}
\C_{a,b,c}^\omega&=\{\psi_\omega\in\mF\mid \frac{1}{\mu_\omega(A_\omega)}\int_{A_\omega}\psi_\omega dm_\omega \le a\int\psi_\omega dm_\omega; \\
&\hskip 2cm|\psi_\omega|_h\le b \int \psi_\omega dm_\omega; \text{ for }x\in\mathcal G_\omega^c, \quad|\psi_\omega(x)|\le c\int \psi_\omega dm_\omega\}.
\end{split}
\end{equation}  
indexed by positive constants $a,b,c$. We consider the usual projective metric (Hilbert metric) on the cones:
for $\vp,\psi\in\C_{a,b,c}^\omega$ let 
\begin{align*}
&\mathfrak{a}(\vp, \psi)=\inf\{\rho>0\mid \rho\vp-\psi\in \C^\omega_{a,b,c}\},\\
&\mathfrak{b}(\vp, \psi)=\sup\{\zeta>0\mid \psi-\zeta\vp\in \C^\omega_{a,b,c}\},\\
&\Theta(\vp, \psi)=\log \left(\frac{\mathfrak{a}(\vp, \psi)}{\mathfrak{b}(\vp, \psi)}\right).
\end{align*}
Further define the cone of positive functions: 
$\C_+^{\omega}=\{\psi_\omega\in\mF\mid \psi_\omega>0\} $
with the corresponding projective metric:
\[
\Theta^+(\vp, \psi)=\log\frac{\sup \vp}{\inf \psi}\cdot\frac{\sup \psi}{\inf \vp}.
\]

Our goal is first to prove that for some $k$ (which will depend on $\omega$ in our current setting) $P^k_\omega\C^\omega(a, b, c)\subseteq\C^{\sigma^k\omega}(a, b, c)$, and to show that $P^k_\omega\C^\omega(a, b, c)$ has finite diameter, with respect to the Hilbert metric, in $\C^{\sigma^k\omega}(a, b, c)$. Then using, the result of Birkhoff \cite{Birk1,Birk2}, we obtain
\begin{equation}\label{eq:dist1}
\Theta\left(P^k_\omega\varphi_\omega,P^k_\omega\psi_\omega\right)\le\text{tanh}\left(\frac{d}{4}\right)\Theta(\varphi_{\sigma^k\omega},\psi_{\sigma^k\omega})\quad \text{ for all } \varphi_\omega, \psi_\omega\in \C^{\omega},
\end{equation}
 where $d$ is the diameter, with respect to the Hilbert metric, of $P^k_\omega\C^\omega(a, b, c)$. We stress that unlike the deterministic setting \cite{Liv1, Liv2, MD01} where the exponential decay of correlation becomes a consequence, in the current random setting (since the $k$ will depend on $\omega$) more effort is needed, see Remark \ref{rem:goodins} below.

Now, recall $\alpha=\frac{1}{2C^2}$, $\alpha'=\frac{3C^2}{2}$ from Lemma \ref{lem:mix}, where $1/C\le d\mu_\omega/d\leb_\omega\le C$. Set $0<\kappa<e^{-\theta'}<1$. We fix the following constants: 
\[
a> \frac1\kappa(\alpha'+\frac\alpha2), \;  c> \frac{C(4\alpha'\D a+\alpha)}{4\alpha'(\kappa-D_F\D_2)},
\; b>\frac{D_FCc}{\kappa-\eps}.
\]
The starting point to obtain contraction for the transfer operators is the following proposition, which claims that the family of cones is $\PP$-almost everywhere eventually preserved by transfer operators. More precisely: 

\begin{proposition}\label{prop:key}
There exists a random  variable $q_1:\Omega\to \NN$ such that for any $\omega\in\Omega$,  $k\ge q_1(\omega)$
\[
P^{k}_\omega\C^\omega(a,b,c)\subset\C^{\sigma^k\omega}(\kappa a,\kappa b, \kappa c).
\]
Moreover, if $\psi_{\sigma^k\omega}\in P^k_\omega\C^\omega(a,b,c)$ then there exists $L>0$ such that 
\[
\frac{1}{\mu_{\sigma^k\omega}(A_{\sigma^k\omega})}\int \psi_{\sigma^k\omega} dm_{\sigma^k\omega}\ge L\quad \text{ for all } \quad A_{\sigma^k\omega}\in \hat\mP_{\sigma^k\omega}.
\]
\end{proposition}
\begin{proof}
For the first condition we have
\begin{equation}\label{eq:cond2}
    \begin{split}
        &\frac{1}{\mu_{\sigma^k\omega}(A_{\sigma^k \omega})}\int_{A_{\sigma^k\omega}}P^k_\omega\psi_\omega dm_{\sigma^k\omega}= \frac{1}{\mu_{\sigma^k\omega}(A_{\sigma^k \omega})}\int_{F_\omega^{-k}A_{\sigma^k\omega}}\psi_\omega dm_\omega\\
        &=\frac{1}{\mu_{\sigma^k\omega}(A_{\sigma^k \omega})}\sum_{A'_\omega\in\bar\mP_\omega}\int_{F_\omega^{-k}A_{\sigma^k\omega}\cap A'_\omega}\psi_\omega dm_\omega+ \frac{1}{\mu_{\sigma^k\omega}(A_{\sigma^k \omega})}\int_{F_\omega^{-k}A_{\sigma^k\omega}\cap \mathcal G_\omega^c}\psi_\omega dm_\omega\\
        &= (I)+(II).
    \end{split}
\end{equation}
Notice that for $x,y\in A'_\omega$,
$$\psi_\omega(y)-|\psi_\omega|_hd(x,y)\le\psi_\omega(x)\le \psi_\omega(y)+|\psi_\omega|_hd(x,y).$$
Integrating with respect to $y$ on $A'_{\omega}$, we get
\begin{equation}\label{eq:vero}
\int_{A'_{\omega}}\psi_\omega(y)dm_\omega-m_\omega(A'_{\omega})|\psi_\omega|_h\D_1\le m_\omega(A'_{\omega})\psi_\omega(x)\le \int_{A'_{\omega}}\psi_\omega(y)dm_\omega+m_\omega(A'_{\omega})|\psi_\omega|_h\D_1,    
\end{equation}
where $\D_1$ is the constant defined in \eqref{eq:dconst}. Let us consider a summand of $(I)$
$$\frac{1}{\mu_{\sigma^k\omega}(A_{\sigma^k \omega})}\int_{F_\omega^{-k}A_{\sigma^k\omega}\cap A'_\omega}\psi_\omega dm_\omega\le \frac{m_\omega(F_\omega^{-k}A_{\sigma^k\omega}\cap A'_\omega)}{\mu_{\sigma^k\omega}(A_{\sigma^k\omega})m_\omega (A'_\omega)}\left[\int_{A'_{\omega}}\psi_\omega(y)dm_\omega+m_\omega(A'_{\omega})|\psi_\omega|_h\D_1\right].$$
By \eqref{eq:mix} we obtain
$$\frac{1}{\mu_{\sigma^k\omega}(A_{\sigma^k \omega})}\int_{F_\omega^{-k}A_{\sigma^k\omega}\cap A'_\omega}\psi_\omega dm_\omega\le \alpha'\left[\int_{A'_{\omega}}\psi_\omega(y)dm_\omega+m_\omega(A'_{\omega})|\psi_\omega|_h\D_1\right] $$
and similarly, for the lower bound, we have
$$\frac{1}{\mu_{\sigma^k\omega}(A_{\sigma^k \omega})}\int_{F_\omega^{-k}A_{\sigma^k\omega}\cap A'_\omega}\psi_\omega dm_\omega\ge \alpha\left[\int_{A'_{\omega}}\psi_\omega(y)dm_\omega-m_\omega(A'_{\omega})|\psi_\omega|_h\D_1\right].$$
Using the above inequalities, the second cone condition and summing over $A'_\omega$ we get
\begin{equation}\label{eq:ess1}
\alpha\left[\sum_{A'_\omega\in\bar\mP_\omega}\int_{A'_{\omega}}\psi_\omega dm-\D_1b\int \psi_\omega dm_\omega \right]\le (I)\le \alpha'\left[\sum_{A'_\omega\in\bar\mP_\omega}\int_{A'_{\omega}}\psi_\omega dm_\omega+\D_1b\int \psi_\omega dm_\omega \right].
\end{equation}
For $(II)$ by using the third cone condition, \eqref{eq:mix} and the constant $\D_2$ from \eqref{eq:dconst}, we have
$$\frac{1}{\mu_{\sigma^k\omega}(A_{\sigma^k \omega})}\int_{F_\omega^{-k}A_{\sigma^k\omega}\cap \bar\G^c_\omega}\psi_\omega dm_\omega\le \frac{m_\omega(F_\omega^{-k}A_{\sigma^k\omega}\cap \bar\G^c_\omega)}{\mu_{\sigma^k\omega}(A_{\sigma^k \omega})m_\omega (\G^c_\omega)}\D_2 c\int\psi_\omega dm_\omega\le \alpha'\D_2 c\int\psi_\omega dm_\omega.$$
Moreover, for $x\in \G^c_\omega$, by integrating the inequality
$$\psi_\omega(x)\ge\frac{1}{m_\omega(\G^c_\omega)}\int_{\G^c_\omega}\psi_\omega dm_\omega-2\sup_{x\in \G^c_\omega}|\psi_\omega(x)|$$
and using similar arguments as in the above inequality we get
$$\frac{1}{\mu_{\sigma^k\omega}(A_{\sigma^k \omega})}\int_{F_\omega^{-k}A_{\sigma^k\omega}\cap \G^c_\omega}\psi_\omega dm_\omega\ge \alpha\int_{\G^c_\omega}\psi_\omega dm_\omega-2\alpha'\D_2c\int\psi_\omega dm_\omega.$$
Using the above inequalities, we get
\begin{equation}\label{eq:ess2}
\alpha\int_{\G^c_\omega}\psi_\omega dm_\omega-2\alpha'\D_2c\int\psi_\omega dm_\omega\le (II)\le   \alpha'\D_2 c\int\psi_\omega dm_\omega.  
\end{equation}
Using \eqref{eq:ess1}, \eqref{eq:ess2} in \eqref{eq:cond2} and duality \eqref{eq:duality} of $P_\omega$, we get
\begin{equation}\label{eq:finala}
\begin{split}
&(\alpha- \alpha'\D_1b-2\alpha'c\D_2) \int P^k_\omega\psi_\omega dm_{\sigma^k\omega}  \le\frac{1}{\mu_{\sigma^k\omega}(A_{\sigma^k \omega})}\int_{A_{\sigma^k\omega}}P^k_\omega\psi_\omega dm_{\sigma^k\omega}  \\
&\hskip 3cm\le (1+b\D_1+c\D_2)\alpha' \int P^k_\omega\psi_\omega dm_{\sigma^k\omega}. 
\end{split}
\end{equation}
To verify the second cone condition we use \eqref{eq:vero} and \eqref{eq:LYC}. Indeed, by \eqref{eq:vero} for $x\in A_\omega'\in\bar\mP_\omega$ and the second cone condition we have
\begin{equation}\label{eq:vero2}
    \psi_\omega(x)\le \frac{1}{ m_\omega(A'_{\omega})}\int_{A'_{\omega}}\psi_\omega(y)dm_\omega+|\psi_\omega|_h\D_1
\le [\D a+\D_1b]\int\psi_\omega dm_\omega,
\end{equation}
and consequently, by the third cone condition we get
$$\psi_\omega(x)\le \max\{\D a+\D_1b,c\}\int\psi_\omega dm_\omega,$$
where $\D=\esssup_\omega\sup_{A_\omega\in\bar\mP_\omega} \frac{\mu_\omega(A_\omega)}{m_\omega(A_\omega)}$. Now, this, \eqref{eq:LYC} and duality \eqref{eq:duality} imply

\begin{equation}\label{eq:tabarnak}
 \frac{|P_\omega^k\psi(x) - P_\omega^k\psi(y)|}{d(x,y)}\le  
 \begin{cases}
 \varepsilon b\int P_\omega^k\psi_\omega dm_{\sigma^k\omega} \text{ for } \ell\ge k;\\
 \varepsilon (b+\max\{\D a+\D_1b,c\}) \int P_\omega^k \psi_\omega dm_{\sigma^k\omega}  \text{ for } N\le 2\ell< k;\\
 \left(\varepsilon b +D_FC\max\{\D a+\D_1b,c\}\right) \int P_\omega^k \psi_\omega dm_{\sigma^k\omega} \text{ for }  2\ell< N.
 \end{cases}
\end{equation}
Finally, we verify the third cone condition. For $x\in\mathcal G_{\sigma^k\omega}^c$ and for $k\le\ell$ notice that $F_\omega^{-k}(x)\in \G^c_{\omega}$. Therefore, \begin{equation}\label{eq:four1}
|P^k_\omega\psi_\omega(x)|=e^{-\theta'\ell}\psi_\omega(F_\omega^{-k}(x))\le e^{-\theta'\ell}c\int \psi_\omega dm_\omega.
\end{equation}
For $k> \ell$, let $q=k-\ell$. Let $x\in\Delta_{\sigma^k\omega,\ell}\cap\G^c_{\sigma^k\omega}$ and $x_0=F^{-\ell}_{\sigma^q\omega}x$. We have 
\begin{equation}\label{eq:four2}
\begin{split}
|P^q_\omega\psi_\omega(x)|=e^{-\theta'\ell}\left(\sum_{\underset{x'\in\G^c_\omega}{F^q_\omega(x')=x_0}}\frac{v(x')\psi_\omega(x')}{JF^q_\omega(x')}+\sum_{\underset{x'\notin\G^c_\omega}{F^q_\omega(x')=x_0}}\frac{v(x')\psi_\omega(x')}{JF^q_\omega(x')}\right).\\
(I)+(II).
\end{split}
\end{equation}
Notice that $(I)$ can be estimated by
$$c\int\psi_\omega dm_\omega\sum_{\underset{x'\in\G^c_\omega}{C_q(x')}}D_F v_\omega(x')\leb_\omega(C_q(x')) \le c\D_2 D_F \int\psi_\omega dm_\omega.$$
By \eqref{eq:vero2} and \eqref{eq:Uni} $(II)$ can be bounded by
$$C(a\D+b\D_1)\int\psi_\omega dm_\omega.$$
Using the above estimates in \eqref{eq:four2}, we get
$$|P^q_\omega\psi_\omega(x)|\le e^{-\theta'\ell}[C(a\D+b\D_1)+c\D_2 D_F ]\int\psi_\omega dm_\omega.$$
Consequently,
\begin{equation}\label{eq:four3}
|P^k_\omega\psi_\omega(x)|\le[C(a\D+b\D_1)+c\D_2 D_F ]\int P^k_\omega\psi_\omega dm_{\sigma^k\omega}.
\end{equation}
Recall $\varepsilon$ in \eqref{eq:tabarnak} which is arbitrarily small. We now fix $\varepsilon  <\kappa$ and  choose $k\ge q_1(\omega)=\max\{k_0, q_0(\omega)\}$, where $q_0$ is as in Lemma \ref{lem:mix} and $k_0$ is from Corollary \ref{cor:LY}, large enough so that 
\begin{equation}\label{eq:1st}
\alpha'\D_1b<\alpha/4; \quad\alpha'c\D_2<\alpha/4;\quad \text{ and }\quad  D_F\D_2<\kappa.
\end{equation}
Then, the lower bound of \eqref{eq:finala}, for all $A_{\sigma^k\omega}\in \hat\mP_{\sigma^k\omega}$ we have
$$\frac{1}{\mu_{\sigma^k\omega}(A_{\sigma^k\omega})}\int \psi_{\sigma^k\omega} dm_{\sigma^k\omega}\ge L :=\alpha- \alpha'\D_1b-2\alpha'c\D_2>0.$$
Using the upper bound of \eqref{eq:finala} and \eqref{eq:1st}, choose $a$ so that $(1+b\D_1+c\D_2)\alpha'\le \kappa a$, we obtain
\begin{equation}\label{eq:2nd}
    a>\frac{1}{\kappa}(\alpha'+\frac{\alpha}{2}).
\end{equation}
Next, using \eqref{eq:four3} and \eqref{eq:1st} choose $c D_F\D_2+C(\D a+\D_1b)< \kappa c$; i.e.,

\begin{equation}\label{eq:3rd}
c>\frac{C(\D a+\D_1b)}{\kappa- D_F\D_2}.
\end{equation}
Finally, using \eqref{eq:tabarnak} and \eqref{eq:3rd} choose $\varepsilon b +D_FCc< \kappa b$; i.e.,
$$ b>\frac{D_FCc}{\kappa-\varepsilon}.$$
This completes the proof of the proposition.
\end{proof}

Next we show that $P^k_\omega\C^\omega(a, b, c)$ has finite diameter, with respect to the Hilbert metric, in $\C^{\sigma^k\omega}(a, b, c)$. 
\begin{lemma}\label{le:finite_diam}
Let  $q_1:\Omega\to \NN$ be as in Proposition \ref{prop:key}. For  $k\ge q_1(\omega)$ the projective diameter of $P_\omega^k\C^{\omega}(a,b,c)$ in $\C^{\sigma^k\omega}(a,b,c)$ is bounded by 
\[
\log\frac{d}{\min\{d, 1-\kappa\}}, 
\]
where $d=\max\left\{\frac{4\alpha'+2\alpha}{\alpha}, \frac{1+\kappa}{1-\kappa}\right\}$, $\alpha, \alpha'$ are as in Lemma \ref{lem:mix}, and $\kappa<e^{-\theta'}$.
\end{lemma}
\begin{proof}
Let $\vp_{\sigma^k\omega}, \psi_{\sigma^k\omega}\in P^k_\omega\C^\omega(a, b, c)$ and $\varrho>0$ be such that $\varrho\vp_{\sigma^k\omega}-\psi_{\sigma^k\omega} \in \C^{\sigma^k\omega}(a, b, c)$.   Then 
\[
0<\frac{1}{\mu_{\sigma^k\omega}(A_{\sigma^k\omega})}\int_{A_{\sigma^k\omega}} (\varrho\vp_{\sigma^k\omega}-\psi_{\sigma^k\omega}) dm_{\sigma^k\omega}\le \int (a\varrho\vp_{\sigma^k\omega}-a\psi_{\sigma^k\omega}) dm_{\sigma^k\omega}.
\]
The above inequalities are satisfied if we choose $\varrho$ so that
\begin{equation*}
\varrho\ge \sup_{A_{\sigma^k\omega}\in \bar\mP_{\sigma^k\omega}}\frac{\int_{A_{\sigma^k\omega}} \psi_{\sigma^k\omega} dm_{\sigma^k\omega}}{\int_{A_{\sigma^k\omega}} \vp_{\sigma^k\omega}dm_{\sigma^k\omega}};\quad 
\varrho\ge \frac{a\int \psi_{\sigma^k\omega} dm_{\sigma^k\omega} -1/\mu_{\sigma^k\omega}(A_{\sigma^k\omega})\int_{A_{\sigma^k\omega}}\psi_{\sigma^k\omega}dm_{\sigma^k\omega}}{ a\int \vp_{\sigma^k\omega} dm_{\sigma^k\omega} -1/\mu_{\sigma^k\omega}(A_{\sigma^k\omega})\int_{A_{\sigma^k\omega}}\vp_{\sigma^k\omega}dm_{\sigma^k\omega}}.
\end{equation*}
To obtain lower bound for the first term we substitute  \eqref{eq:1st} and \eqref{eq:2nd} into \eqref{eq:finala} we obtain 
\begin{equation}\label{eq:rho21}
\varrho \ge \frac{4\alpha'+2\alpha}{\alpha}\frac{\int \psi_{\sigma^k\omega} dm_{\sigma^k\omega} }{\int \vp_{\sigma^k\omega} dm_{\sigma^k\omega} }.
\end{equation}
 Since $\psi\ge 0$ we have 
 \[
 a\int \psi_{\sigma^k\omega} dm_{\sigma^k\omega} -1/\mu_{\sigma^k\omega}(A_{\sigma^k\omega})\int_{A_{\sigma^k\omega}}\psi_{\sigma^k\omega}dm_{\sigma^k\omega}\le a \int \psi_{\sigma^k\omega} dm_{\sigma^k\omega}. 
 \]
 Using again by \eqref{eq:finala} and \eqref{eq:2nd}
 \[
 a\int \vp_{\sigma^k\omega} dm_{\sigma^k\omega} -1/\mu_{\sigma^k\omega}(A_{\sigma^k\omega})\int_{A_{\sigma^k\omega}}\vp_{\sigma^k\omega}dm_{\sigma^k\omega}\ge a(1-\kappa)\int \vp_{\sigma^k\omega} dm_{\sigma^k\omega}.
 \]
 Thus, 
\begin{equation}\label{eq:rho22}
\varrho \ge \frac{1}{1-\kappa}\frac{\int \psi_{\sigma^k\omega} dm_{\sigma^k\omega} }{\int \vp_{\sigma^k\omega} dm_{\sigma^k\omega} }.
\end{equation}
For the third cone condition we need
\begin{equation}\label{eq:cont3}
\begin{split}
-b \int_{A_{\sigma^k\omega}} (\varrho\vp_{\sigma^k\omega}-\psi_{\sigma^k\omega}) dm_{\sigma^k\omega}\le &
\frac{\varrho(\vp_{\sigma^k\omega}(x)-\vp_{\sigma^k\omega}(y))-(\psi_{\sigma^k\omega}(x)-\psi_{\sigma^k\omega}(y))}{d_{\sigma^k\omega}(x, y)}\\
&\quad\le b \int_{A_{\sigma^k\omega}} (\varrho\vp_{\sigma^k\omega}-\psi_{\sigma^k\omega}) dm_{\sigma^k\omega}.
\end{split}
\end{equation}

Since 
$$\frac{|\varrho(\vp_{\sigma^k\omega}(x)-\vp_{\sigma^k\omega}(y))|+|(\psi_{\sigma^k\omega}(x)-\psi_{\sigma^k\omega}(y))|}{d_{\sigma^k\omega}(x, y)}
\quad\le \kappa b \int_{A_{\sigma^k\omega}} (\varrho\vp_{\sigma^k\omega}+\psi_{\sigma^k\omega}) dm_{\sigma^k\omega}.$$
Thus, comparing this with the right hand side of \eqref{eq:cont3}, we choose $\varrho$ to satisfy
$$\kappa b\int_{A_{\sigma^k\omega}}(\varrho\vp_{\sigma^k\omega}+\psi_{\sigma^k\omega})\le b \int_{A_{\sigma^k\omega}} (\varrho\vp_{\sigma^k\omega}-\psi_{\sigma^k\omega});$$
i.e.,
\begin{equation}\label{eq:34rho}
\varrho\ge\frac{1+\kappa}{1-\kappa}\frac{\int \psi_{\sigma^k\omega} dm_{\sigma^k\omega} }{\int \vp_{\sigma^k\omega} dm_{\sigma^k\omega} }.
\end{equation}
For the fourth cone condition, we need for any  $x\in \G^c_{\sigma^k\omega}$
\begin{equation}\label{eq:cont4}
|(\varrho\vp_{\sigma^k\omega}-\psi_{\sigma^k\omega})(x)|\le c\int \varrho\vp_{\sigma^k\omega}-\psi_{\sigma^k\omega} dm_{\sigma^k\omega}.
\end{equation}
Again, comparing the following estimate with the right hand side of \eqref{eq:cont4}
$$|\varrho\vp_{\sigma^k\omega}(x)|+|\psi_{\sigma^k\omega}(x)|\le \kappa c\varrho\int \vp dm_{\sigma^k\omega}+ c\kappa\psi dm_{\sigma^k\omega}$$
it is enough to choose $\varrho$ as in \eqref{eq:34rho}.
Thus, by \eqref{eq:rho21}, \eqref{eq:rho22} and \eqref{eq:34rho}, we choose $$\varrho\ge\max\Big\{\frac{4\alpha'+2\alpha}{\alpha}, \frac{1+\kappa}{1-\kappa},\Big\}$$
Now similarly, for $\vp_{\sigma^k\omega}, \psi_{\sigma^k\omega}\in P^k_\omega\C^\omega(a, b, c)$ we obtain $\varsigma>0$ such that $\psi_{\sigma^k\omega}-\varsigma\vp_{\sigma^k\omega} \in \C^{\sigma^k\omega}(a, b, c)$ 
$$\varsigma\le\min\Big\{\frac{\alpha}{4\alpha'+2\alpha}, \frac{1-\kappa}{1+\kappa}, 1-\kappa\Big\}.$$
\end{proof}
\subsection{Concatenations of transfer operators} 
\begin{remark}\label{rem:goodins}
Even though the transfer operators eventually preserve the family of cone fields and their images have finite diameter (recall Proposition~\ref{prop:key} and Lemma~\ref{le:finite_diam}), one cannot conclude directly that the functions $P^k_{\sigma^{-k}\omega}(\psi_{\sigma^{-k}\omega})$, where $\psi_{\sigma^{-k} \omega}\in {\mathcal C}^{\sigma^{-k}\omega}(a,b,c)$,  converge exponentially fast to $h_\omega$ as $k\to+\infty$. The reason being that the sequence of instants at which one observes the contraction (determined by the random variable $q_1(\cdot)$
in Proposition~\ref{prop:key}) could be sparse along the orbits of $\omega$ by the base map. 
In the following, we show that this is not the case.
\end{remark}
As $q_1(\omega)$ is finite for $\PP$-almost every $\omega$,
for every $\eps>0$ there exists $M=M(\eps)\in\NN$ such  that
\begin{equation}\label{eq:finiteae}
\PP(\{\omega: q_1(\omega)>M\})<\eps.
\end{equation}
\begin{lemma}\label{lem:erg} 
There exists an integer valued random variable $0\le r(\omega) \le M$ such that 
\[
\lim_{n\to\infty}\frac{1}{n}\sum_{k=0}^{n-1}{\bf 1}_{\{q_1 \le M\}}\circ \sigma^{Mk+r}\omega\ge (1-\eps)
\]
for $\PP$-almost every $\omega\in \Omega$.
\end{lemma}

\begin{proof}
Since $(\sigma, \PP)$ is ergodic, by the choice of $\eps$ and $M$ in ~\eqref{eq:finiteae} we have 
\begin{equation}\label{eq:eps/2}
\lim_{n\to \infty}\frac{1}{Mn}\sum_{j=0}^{Mn-1}{\bf 1}_{\{q_1\le M\}} (\sigma^j\omega) \ge (1-{\eps})
\end{equation}
for $\PP$-almost every $\omega\in \Omega$.
Since, $\sigma^M$ preserves $\PP$,  the limit 
\[
\lim_{n\to \infty}\frac{1}{n}\sum_{k=0}^{n-1}{\bf 1}_{\{q_1\le M\}}\circ \sigma^{Mk+r_0}\omega
\]
exists for $\PP$-almost every $\omega\in\Omega$ and every integer $0 \le r_0 \le M$. Thus, by writing, 
\[
\frac{1}{Mn}\sum_{j=0}^{Mn-1}{\bf 1}_{\{q_1\le M\}}\circ \sigma^j\omega=
\frac{1}{M}\sum_{r=0}^{M-1} \Big[ \frac{1}{n}\sum_{k=0}^{n-1}{\bf 1}_{\{q_1\le M\}}\circ \sigma^{Mk+r}\omega\Big]
\]
we conclude, for $\PP$-almost every $\omega\in \Omega,$ the existence of $0\le r(\omega) \le M$ with the desired property. 
\end{proof}

Define recursively the following sequence: 
\begin{align*}
&t_1(\omega)= 
\min\{jM+r(\omega)\ge {q_1(\omega)}\mid q_1(\sigma^{jM+r(\omega)}\omega)\le M\},\\
&t_i(\omega)=\min\{jM+r(\omega)>t_{i-1}(\omega)\mid q_1(\sigma^{{jM+r(\omega)}}\omega)\le M\}, \text{ for } i\ge 2.
\end{align*}
Notice that $t_i\ge i$ by definition, hence, $t_i\to \infty $ as $i\to \infty$.

\begin{lemma}\label{lem:erg2}

For $\PP$-almost every $\omega\in \Omega$ there exists $q_3(\omega)\in \NN$ such that for each  $n\ge q_3(\omega)$ the following holds: if $t_1(\omega) < t_2(\omega) < \dots < t_s(\omega)\le n < t_{s+1}(\omega)$ then 
$s\ge \Big\lfloor\frac{n}{M} \Big\rfloor(1-2\eps)$.
\end{lemma}
\begin{proof}
By Lemma \ref{lem:erg} there exists $q_3(\omega) \ge q_1(\omega)$ such that for all  $n\ge q_3(\omega)$ we have
that 
\[
\sum_{k=0}^{\lfloor n/M\rfloor-1}{\bf 1}_{\{q_1\le M\}}\circ \sigma^{Mk+r(\omega)}(\omega) \ge (1-2\eps) \Big\lfloor\frac{n}{M} \Big\rfloor. 
\]
By definition of the $t_i$'s this implies $s\ge (1-2\eps) \Big\lfloor\frac{n-q_1(\omega)}{M} \Big\rfloor$.
\end{proof}

Let $n$ be sufficiently large integer (larger that $q_3$ in the previous lemma). Let $s=\max\{i\in\NN\mid t_i\le n\}$ Then we have 
\[
P^n_\omega=P^{n-t_{s}}_{\sigma^{t_{s}}\omega}\circ P_{\sigma^{t_{s-1}}\omega}^{t_s-t_{s-1}}\circ\dots\circ P_{\sigma^{t_1}\omega}^{t_2-t_1}\circ P_\omega^{t_1}.
\]
By the definition of $t_i$'s we have  $t_i(\omega)-t_{i-1}(\omega)\ge q_1(\sigma^{t_{i-1}}\omega)$. Using Proposition \ref{prop:key} repeatedly, we obtain 
\begin{equation}\label{eq:contractions}
P^{t_{s}}_\omega\C^\omega(a,b,c)\subset\C^{\sigma^{t_{s}}\omega}( a\kappa^{(1-2\eps)\frac{n}{M}}, b\kappa^{(1-2\eps)\frac{n}{M}}, c\kappa^{(1-2\eps)\frac{n}{M}}).
\end{equation}

Notice that $P^{n-t_{s}}_{\sigma^{t_{s}}\omega}$ does not necessarily preserve the cone, but this can be dealt with separately. We are now in a position to study the `operational' correlations on the random tower, first using the measures $\{m_\omega\}$ in Lemma \ref{lem:opcorm}. We later we deduce the same type of correlations on the tower using Lebesgue measure $\{\leb_\omega\}$.
\subsection{Decay of operational correlations}
We start with a standard lemma (see \cite{Liv2} for example) that will be used to relate the correlation the underlying system with the contraction of the corresponding transfer operator in the Hilbert metric.
\begin{lemma}\label{lem:inftyP}
For $\varphi_\omega,\psi_\omega \in C^\omega_+$ with $\|\varphi_\omega\|_1=1$ and $\|\psi_\omega\|_1=1$ we have 
\[\left\|\frac{\varphi_\omega}{\psi_\omega}-1\right\|_\infty
\le \left(\exp{\Theta^+\left( \varphi_\omega,\psi \right)}-1\right),
\]
\end{lemma}
\begin{proof}
Recall that the definition of $\Theta^+$ implies 
\[
\Theta^+(\vp_\omega, \psi_\omega)=\log\left(\sup_{x,y,\in\Delta_\omega}\frac{\vp_\omega(x)\psi_\omega(y)}{\vp_\omega(y)\psi_\omega(x)}\right).
\]
Thus, using the identity 
\[
\frac{\varphi_\omega(x)}{\psi_\omega(x)}=\frac{\varphi_\omega(x)\psi_\omega(y)}{\vp_\omega(y)\psi_\omega(x)}\cdot\frac{\vp_\omega(y)}{\psi_\omega(y)},
\]
we obtain 
\[
e^{-\Theta^+(\vp_\omega, \psi_\omega)}\frac{\vp_\omega(y)}{\psi_\omega(y)}\le \frac{\vp_\omega(x)}{\psi_\omega(x)}\le e^{\Theta^+(\vp_\omega, \psi_\omega)}\frac{\vp_\omega(y)}{\psi_\omega(y)}.
\]
Since $\|\psi\|_1=\|\vp\|_1$ there exists two points $y_1, y_2 \in\Delta_\omega$ such that $\psi(y_1)\ge \vp(y_1)$ and $\vp(y_2)\ge \psi(y_2)$, so we obtain 
\[
e^{-\Theta^+(\vp_\omega, \psi_\omega)}\le \frac{\vp_\omega(x)}{\psi_\omega(x)}\le e^{\Theta^+(\vp_\omega, \psi_\omega)},
\]
which completes the proof. 
\end{proof}
We first study the correlations for observables in the family of cones.
\begin{lemma}\label{lem:opcorm}
Let $\tilde\varphi_\omega\in\C^\omega(a,b,c)$, $\varphi_\omega=v_\omega\tilde\varphi_\omega$ and $\psi_\omega\in L^{\infty}$. There exists a random variable $C_\omega(\varphi,\psi)>0$ finite $\PP$-almost everywhere and a constant $0<\rho<1$ such that for every $n\in\NN$
$$\left| \int_{\Delta_{\omega}} (\psi_{\sigma^n\omega}\circ F^{n}_\omega)\varphi_\omega d\leb_{\omega}-\int_{\Delta_{\sigma^n\omega}} \psi_{\sigma^n\omega} d\mu_{\sigma^n\omega}\int_{\Delta_\omega} \varphi_\omega d\leb_\omega\right|
\le C_\omega(\varphi,\psi)\rho^n.$$
\end{lemma}

\begin{proof}
We first prove the above for $n\ge q_3(\omega)$, where $q_3$ is the same random variable as Lemma \ref{lem:erg2}. Assume without loss of generality that  $\int_{\Delta_\omega}\varphi_\omega d\leb_\omega=1$. Consequently, the operational correlation in the lemma is bounded above by:
$$\|\psi_{\sigma^n\omega}\|_{\infty}\cdot\int_{\Delta_{\sigma^n\omega}}\left|\mL_\omega^n\varphi_\omega-h_{\sigma^n\omega}\right|d\leb_{\sigma^n\omega}.$$
Therefore to conclude the lemma, we estimate the integral in the above expression.
Using the fact that $\mL_\omega$ is a weak contraction in $L^1(\Delta_\omega)$ and equation \eqref{eq:contractions}, we have

\begin{equation}
    \begin{split}
\int_{\Delta_{\sigma^n\omega}}\left|\mL_\omega^n\varphi_\omega-h_{\sigma^n\omega}\right|d\leb_{\sigma^n\omega}&=\int_{\Delta_{\sigma^n\omega}}\left|\mL_\omega^{n-t_s+t_s}\varphi_\omega -\mL_\omega^{n-t_s+t_s} h_{\omega}\right|d\leb_{\sigma^n_\omega}\\
&\le \int_{\Delta_{\sigma^{t_s}_\omega}}\left|\mL_\omega^{t_s}\varphi_\omega-\mL_\omega^{t_s} h_{\omega}\right|d\leb_{\sigma^{t_s}\omega}\\
&\le \int_{\Delta_{\sigma^{t_s}\omega}}\left|\mL_\omega^{t_s}\varphi_\omega-\mL_\omega^{t_s} h_{\omega}\right|\frac{v_{\sigma^{t_s}\omega}}{v_{\sigma^{t_s}\omega}}d\leb_{\sigma^{t_s}\omega}\\
&=\int_{\Delta_{\sigma^{t_s}\omega}}\left|P_\omega^{t_s}\tilde\varphi_\omega-\tilde h_{\sigma^{t_s}\omega}\right|dm_{\sigma^{t_s}_\omega}\\
&\le\|\tilde h_{\sigma^{t_s}\omega}\|_\infty\cdot\left\|\frac{P_\omega^{t_s}\tilde\varphi_\omega}{\tilde h_{\sigma^{t_s}\omega}}-1\right\|_\infty\\
&\le C \left(\exp\{\Theta^+\left( P_\omega^{t_s}\tilde \varphi_\omega,\tilde h_{\sigma^{t_s}\omega}\right)\}-1\right)\\
&\le C\left(\exp\{\Theta\left( P_\omega^{t_s}\tilde \varphi_\omega,\tilde h_{\sigma^{t_s}\omega}\right)\}-1\right)\\
&\le C\left(\exp\{\kappa^{(1-2\eps){\Big\lfloor\frac{n-q_1(\omega)}{M}}\Big\rfloor} \}-1\right)\le C \rho^n.
\end{split}
\end{equation}
where we used Lemma \ref{lem:inftyP} to pass from the $L^\infty$ norm to the projective norm. 
Now for $n\ge 0$ we have
$$\left| \int_{\Delta_{\omega}} (\psi_{\sigma^n\omega}\circ F^{n}_\omega)\varphi_\omega d\leb_{\omega}-\int_{\Delta_{\sigma^n\omega}} \psi_{\sigma^n\omega} d\mu_{\sigma^n\omega}\int_{\Delta_\omega} \varphi_\omega d\leb_\omega\right|
\le C_\omega(\varphi,\psi)\rho^n,$$
where $C_\omega=2C\rho^{-q_3(\omega)}\cdot\|\psi_{\sigma^n\omega}\|_{\infty}$.

\end{proof}

Below we show that every $\psi\in\mF$ can be modified without affecting the correlations, so that modified observable is in $\C^\omega_{a,b,c}$. Hence, to obtain decay of correlations and completes the proof of Theorem \ref{thm:main1}, we Lemma \ref{lem:opcorm} above and the following lemma.

\begin{lemma}
For every $\varphi\in \mF$ let $K_\omega=\int\varphi_\omega d\leb_{\omega}$. There exists $C_\varphi>0$ such that 
\[
\tilde\varphi_\omega=\frac{1}{v_\omega}\frac{\varphi_\omega+C_\varphi h_\omega}{K_\omega+C_\varphi}\in\C^{\omega}_{a,b,c},
\]
for almost every $\omega\in\Omega$ and for all $a>1$, $b>\sup_\omega|\tilde h_\omega|_h$, $c>\sup_\omega\| \tilde h_\omega\|_\infty$.
\end{lemma}

\begin{proof}
By definition we have $\int \tilde\varphi_\omega dm_\omega=1$. We will choose $C_\varphi$ satisfying the cone conditions. 
$\tilde \varphi\ge 0$ is satisfied when $C_\varphi\ge \sup_{x\in\Delta_\omega}|\varphi_\omega(x)|/h_\omega(x)$. 
We start with the cone condition
\[
\frac{1}{\mu_\omega(A_\omega)}\int_{A_\omega}\tilde\varphi_\omega dm_\omega\le a \int\tilde\varphi_\omega d m_\omega.
\]
which is equivalent to 
\[
\frac{1}{\mu_\omega(A_\omega)}\cdot\frac{1}{K_\omega+C_\varphi}
\int_{A_\omega}(\varphi_\omega+C_\varphi h_\omega)d\leb_\omega \le a,
\]
which reduces to 
\[
\frac{1}{\mu_\omega(A_\omega)}\left(\int_{A_\omega} \varphi_\omega d\leb_\omega+C_\varphi\right)\le a(K_\omega+C_\varphi).
\]
and 
\[
C_\varphi\ge \frac{\frac{1}{\mu(A_\omega)}\int_{A_\omega}\varphi_\omega d\leb_\omega-aK_\omega}{a-\frac{1}{\mu(A_\omega)}\int_{A_\omega}h_\omega d\leb_\omega}, \text{ for all } A_\omega\in\mP_\omega.
\]
The second cone condition states that 
\[
|\tilde\varphi_\omega|_h\le b\int\tilde\varphi_\omega d m_\omega
\]
which holds if we choose $C_\varphi$ satisfying the following inequality
\[
|\varphi_\omega/ v_\omega|_h+C_\varphi |h_\omega/v_\omega|_h\le b(K_\omega+C_\varphi).
\]
Equivalently, 
\[
C_\psi\ge \frac{|\varphi_\omega/v_\omega|_h-bK_\omega}{b-|h_\omega/ v_\omega|_h}.
\]
The third cone condition is on $\mathcal G^c_\omega$
\[
|\tilde\varphi_\omega(x)|\le c\int\tilde\varphi_\omega(x)d m_\omega, x\in\mathcal G^c_\omega,
\]
which is equivalent to 
\[\left|\frac{\varphi_\omega(x)+C_\varphi h_\omega(x)}{v_\omega(x)}\right|\le c(K_\omega+C_\varphi).
\]
The inequality is satisfied if we choose $C_\phi>0$ using 
$$|\varphi_\omega(x)/ v(x)|+C_\varphi |h_\omega(x)/ v_\omega(x)|\le c(K_\omega+C_\varphi).$$ 
Consequently, if we choose 
\[
C_\varphi\ge \frac{\|\varphi_\omega/ v_\omega\|_\infty-cK_\omega}{c-\|h_\omega/ v_\omega\|_\infty}.
\]
the above cone condition is satisfied. 
\end{proof}
\bibliographystyle{amsplain}

\begin{thebibliography}{10}
\bibitem{ABR22} Alves, J. F.; Bahsoun, W.; Ruziboev, M. Almost sure rates of mixing for partially hyperbolic attractors. \emph{J. Differential Equations} 311 (2022), 98--157. 
\bibitem{AV13} Alves, J.F.; Vilarinho, H. Strong stochastic stability for non-uniformly expanding maps. \emph{Ergodic Theory and Dynam. Systems} 33(3), (2013), 647--692. 
\bibitem{AFGV} Atnip J., Froyland G., Gonz\'alez-Tokman C. Vaienti S. Thermodynamic Formalism for Random Weighted Covering Systems. {\em Comm. Math. Phys.}, vol. 386, 819-- 902, 2021.
\bibitem{A98}
Arnold, L.
\newblock {\em Random dynamical systems}.
\newblock Springer Monographs in Mathematics. Springer-Verlag, Berlin, 1998.
\bibitem{BBR19} Bahsoun, W.; Bose, C.; Ruziboev, M. Quenched decay of correlations for slowly mixing systems. \emph{Trans. Amer. Math. Soc.} 372 (2019), no. 9, 6547--6587. 
\bibitem{BBMD} Baladi, V., Benedicks, M., Maume-Deschamps, V.  Almost sure rates of mixing for i.i.d. unimodal maps, \emph{Ann. Sci. \'Ecole Norm. Sup.} Vol. 35 (2002), no. 1, 77 --126.\\
Baladi, V., Benedicks, M., Maume-Deschamps, \emph{V. Ann. Sci. \'Ecole Norm. Sup.} Vol. 36 (2003), 319--322 (corrigendum).\\
Baladi, V., Benedicks, M., Maume-Deschamps, V. \emph{Correcting the proof of Theorem 3.2 and Corollary 5.2 in Almost sure rates of mixing for i.i.d. unimodal maps}. arXiv:1008.5165 (corrigenda).
\bibitem{Birk1} Birkhoff, G. Extensions of Jentzsch's theorem. \emph{Trans. Amer. Math. Soc.} 85 (1957), 219--227.
\bibitem{Birk2} Birkhoff, G. \emph{Lattice Theory}, 3rd edition, American Mathematical Society Colloquium Publications, Vol. XXV, Amer. Math. Soc., Providence, RI, 1967. 
 \bibitem{BR75} Bowen, R., Ruelle, D. The ergodic theory of Axiom A flows. \emph{Invent. Math.} 29(3), 181--202, 1975.
 \bibitem{BLvS}  Bruin, H.; Luzzatto, S.; Van Strien, S. Decay of correlations in one-dimensional dynamics. \emph{Ann. Sci. \'Ecole Norm. Sup.} (4) 36 (2003), no. 4, 621--646.
\bibitem{Bu} Buzzi, J. Exponential decay of correlations for random Lasota-Yorke maps. {\em Comm. Math. Phys.} 208, no. 1, 25--54, 1999.
\bibitem{Cra02}
Crauel, H. \newblock {\em Random probability measures on {P}olish spaces}, volume~11 of
  {\em Stochastics Monographs}.
\newblock Taylor \& Francis, London, 2002.
\bibitem{CF94}
Crauel, H., Flandoli, F.
\newblock Attractors for random dynamical systems.
\newblock {\em Probab. Theory Related Fields}, 100(3):365--393, 1994.
\bibitem{DFGV19}
Dragi{\v c}evi{\'c}, D., Froyland, G, Gonz{\'a}lez-Tokman, C., Vaienti, S.
\newblock A spectral approach for quenched limit theorems for random hyperbolic
  dynamical systems.
\newblock {\em Trans. Amer. Math. Soc.} 373, 629--664, 2020.
\bibitem{DH20} Dragi\v{c}evi\'{c}, D., Hafouta. Y.
\newblock Limit theorems for random expanding or Anosov dynamical systems and vector-valued observables.
\newblock \emph{Ann. Henri Poincar\'e}, 21, 3869--3917, 2020.
\bibitem{DS21} Dragi\v{c}evi\'{c}, D., Sedro, J.
\newblock Quenched limit theorems for expanding on average cocycles. Available at arXiv:2105.00548
\bibitem{H22} Hafouta, Y. Limit theorems for random non-uniformly expanding or hyperbolic maps with exponential tails. \emph{Ann. Henri Poincar\'e} 23 (2022), no. 1, 293--332.
\bibitem{HKRVZ22} Homburg, A.J., Kalle, C., Ruziboev, M., Verbitskiy, E., Zeegers, B. Critical intermittency in random interval maps. \emph{Commun. Math. Phys.} to appear
\bibitem{K} Kifer, Y. \textit{Limit theorems for random transformations and processes in random environments}. Trans. Amer. Math. Soc. 350 (1998), no. 4, 1481--1518.
\bibitem{La22} Larkin, A. Quenched decay of correlations for one-dimensional random Lorenz maps. \emph{J. Dyn. Control Syst.} (2022). https://doi.org/10.1007/s10883-021-09583-w
\bibitem{LQ95} Liu, P-D., Qian, M. Smooth ergodic theory of random dynamical systems. Lecture Notes in Mathematics, 1606. Springer-Verlag, Berlin, 1995.
 \bibitem{Liv1} Liverani, C. Decay of correlations. \emph{Ann. of Math. (2)} 142 (1995), no. 2, 239–301.
\bibitem{Liv2} Liverani, C. Decay of correlations for piecewise expanding maps. \emph{J. Stat. Phys.} 78 (1995), no. 3-4, 1111--1129.
\bibitem{MD01}  Maume-Deschamps, V. Projective metrics and mixing properties on towers. \emph{Trans. Amer. Math. Soc.} 353 (2001), no. 8, 3371--3389.
\bibitem{NPT21} Nicol, M., Pereira, F., T\"or\"ok, A. Large deviations and central limit theorems for sequential and random systems of intermittent maps. \emph{Ergodic Theory Dynam. Systems}. 41(9), (2021) ,2805-2832.
\bibitem{Se} Seneta, E. \textit{Nonnegative matrices and Markov chains}. Second edition. Springer Series in Statistics. Springer-Verlag, New York, 1981.
\bibitem{SSV21} Stadlbauer, M., Suzuki, S., Varandas, P. Thermodynamic formalism for random non-uniformly expanding maps, {\em Comm. Math. Phys.}, 385, 369--427, 2021.
\bibitem{SVZ} Stadlbauer, M., Varandas, P., Zhang, X. Quenched and annealed equilibrium states for random Ruelle expanding maps and applications. Available at arXiv:2004.04763
\bibitem{Su22} Su, Y. Random Young towers and quenched limit laws. \emph{Ergodic Theory Dynam. Systems}. (2022), 1-33. doi:10.1017/etds.2021.164
\bibitem{Y98}
Young, L.-S.
\newblock Statistical properties of dynamical systems with some hyperbolicity.
\newblock {\em Ann. of Math. (2)}, 147(3):585--650, 1998.

\bibitem{Y99}
Young, L.-S.
\newblock Recurrence times and rates of mixing.
\newblock {\em Israel J. Math.}, 110:153--188, 1999.
\end{thebibliography}

\end{document}